\documentclass{amsart}
\usepackage{mathrsfs}
\usepackage{times}
\usepackage{amssymb,amsmath,amscd, amsthm,stmaryrd,verbatim, mathrsfs}

\usepackage{color}
\usepackage{}
\usepackage{bbding}
\usepackage{amsfonts}
\usepackage{amssymb,amscd,colordvi,amsthm,verbatim,mathrsfs,latexsym,  graphicx}

\newtheorem{Thm}{Theorem}[section]
\newtheorem{Prop}{Proposition}[section]
\newtheorem{Lem}{Lemma}[section]

\newtheorem{Rem}{Remark}[section]
\newtheorem{definition}{Definition}[section]
\newtheorem{ex}{Example}[section]
\numberwithin{equation}{section}

\newcommand{\bb}{\mathbb}

\newcommand{\frk}{\mathfrak}

\newcommand{\ptl}{\partial}
\newcommand{\dlt}{\delta}

\newcommand{\lmd}{\lambda}

\newcommand{\mbb}{\mathbb}

\newcommand{\ol}{\overline}

\newcommand{\eq}{\begin{equation}}
\newcommand{\en}{\end{equation}}
\newcommand{\beqna}[1]{\begin{eqnarray}\label{#1}}
\newcommand{\eeqna}{\end{eqnarray}}
\newcommand{\beqn}[1]{\begin{equation}\label{#1}}
\newcommand{\eeqn}{\end{equation}}

\numberwithin{equation}{section}

\usepackage{hyperref}
\begin{document}

\title[Gelfand-Kirillov Dimensions of  the $\bb{Z}$-graded
Oscillator Representations]{ Gelfand-Kirillov Dimensions of  the $\mathbb{Z}$-graded
Oscillator Representations of  $\mathfrak{o}(n,\mathbb{C})$ and $\mathfrak{sp}(2n,\mathbb{C})$}

\author{Zhanqiang Bai}

\address{ School of Mathematics and Statistics , Wuhan University , Wuhan, 430072, P.R. China}
\email{Zhanqiang\.bai@whu.edu.cn}





\maketitle

\begin{abstract}

In this paper, we give a method to compute the Gelfand-Kirillov dimensions of some polynomial type weight modules. These modules are
 infinite-dimensional  irreducible $\mathfrak{o}(n, \bb C)$-modules and $\mathfrak{sp}(2n, \bb C)$-modules that appeared in the
  $\bb{Z}$-graded oscillator generalizations of the classical theorem on harmonic polynomials established by Luo
  and Xu. We also found that some of these modules have the secondly minimal GK-dimension, and some of them have the larger  GK-dimension than the maximal GK-dimension apearing in unitary highest-weight modules.


\end{abstract}

\section{Introduction}

Fifty years ago, Gelfand and Kirillov \cite{Ge-Ki} introduced a quantity to
measure the rate of growth of an algebra in terms of any generating
set, which is now known as  Gelfand-Kirillov dimension. Since then the  Gelfand-Kirillov dimension has become a very useful and powerful tool for people to  measure the size of infinite-dimensional irreducible  modules of Lie algebras and Lie groups. However, usually it is not easy  to compute the Gelfand-Kirillov dimensions of explicit modules.

A module of a finite-dimensional simple Lie algebra is called a {\it
weight module} if it is a direct sum of its weight subspaces. The classification of weight modules had been completed by Mathieu \cite{Ma} after the contributions of many mathematicians.   But we don't have many results about the distribution of Gelfand-Kirillov dimensions of weight modules. Let
$M$ be an irreducible highest-weight module for a finite-dimensional
simple Lie algebra $\mathfrak{g}$. Then $M$ is naturally a weight
module with finite-dimensional weight subspaces.
Denote by $d_{M}$
its Gelfand-Kirillov dimension. We fix a Cartan subalgebra
$\mathfrak{h}$, a root system $\Delta\subset \mathfrak{h}^*$  and a
set of positive roots $\Delta_+ \subset \Delta$. Let $\rho$ be half
the sum of all positive roots. Suppose that $\beta$ is the highest
root. It is well known that $d_M=0$ if and only if $M$ is
finite-dimensional, in which case irreducible modules are classified
by the highest-weight theory. From Vogan \cite{Vogan-81} and Wang
\cite{wang-99}, we know that the next smallest integer occurring is
$d_M=(\rho,\beta^{\vee})$. We call them the \emph{minimal
Gelfand-Kirillov dimension module}. These small modules are of great
interest in representation theory. A general introduction can be
found in Vogan \cite{Vogan-81}. Recently we \cite{Bai-Hu} studied the GK-dimensions of unitary highest-weight modules. We found that the secondly minimal GK-dimension of a unitary highest-weight module is $2((\rho,\beta^{\vee})-C)$ and the maximal GK-dimension of a unitary highest-weight module is $r((\rho,\beta^{\vee})-(r-1)C)$, where $C$ and $r$ are constants only depending on the type of Lie algebras and given by Enright, Howe and Wallach in \cite{EHW}. Does any irreducible weight module have larger GK-dimension than $r((\rho,\beta^{\vee})-(r-1)C)$? We will confirm the answer in this paper.

In classical harmonic analysis,  a  fundamental theorem  says that
the spaces of homogeneous harmonic polynomials are irreducible
modules of the corresponding orthogonal Lie group (algebra) and the
whole polynomial algebra is a free module over the invariant
polynomials generated by harmonic polynomials. Bases of these
irreducible modules can be obtained easily (e.g., cf. \cite{Xu-08}).
The algebraic beauty of the above theorem is that the Laplace equation
characterizes the irreducible submodules of the polynomial algebra
and the corresponding quadratic invariant gives a decomposition of
the polynomial algebra into a direct sum of irreducible submodules,
namely, complete reducibility.
Recently Luo and Xu \cite{Luo-Xu} established the $\bb{Z}^2$-graded
oscillator generalizations of the above theorem for
$\mathfrak{sl}(n,\bb C)$, where the irreducible submodules are $\bb
Z^2$-graded homogeneous polynomial solutions of deformed Laplace
equations. In \cite{Bai-sl}, we find an exact formula of
Gelfand-Kirillov dimensions for these $\mathfrak{sl}(n,\bb C)$-modules. It turns out that their
Gelfand-Kirillov dimensions are independent of the double grading
and three infinite subfamilies of these modules have the minimal
Gelfand-Kirillov dimension. In \cite{Luo-Xu-lie,Luo-Xu-algebra}, by using the results in \cite{Luo-Xu}, Luo and Xu established the structure of the corresponding two-parameter $\mathbb{Z}$-graded oscillator representations of $\mathfrak{o}(n,\mathbb{C})$ and $\mathfrak{sp}(2n,\mathbb{C})$. It turned out that these modules are irreducible weight modules. In this paper, we will compute the Gelfand-Kirillov dimensions of these modules. Below we give  a more detailed  introduction for theses modules.

For convenience, we will use the notion
$\overline{{i,i+j}}=\{i,i+1,i+2,...,i+j\}$ for integers $i$ and $j$
with $i\leq j$. Denote by $\bb{N}$ the additive semigroup of
nonnegative integers. Let $E_{r,s}$ be the square matrix with $1$ as
its $(r,s)$-entry and $0$ as the others.
The orthogonal Lie algebra
$$\mathfrak{o}(2n,\mathbb{C})=\sum_{i,j=1}^n\mathbb{C}(E_{i,j}-E_{n+j,n+i})+\sum_{1\leq
i<j\leq
n}[\mathbb{C}(E_{i,n+j}-E_{j,n+i})+\mathbb{C}(E_{n+j,i}-E_{n+i,j})]$$
Denote
${\mathcal{B}}=\bb{C}[x_1,...,x_n,y_1,...,y_n]$. Fix
 $n_1,n_2\in\overline{1,n}$ with $n_1\leq n_2$. We have the following non-canonical oscillator representation of $\mathfrak{o}(2n,\bb{C})$ on
  ${\mathcal{B}}$ determined by
\eq\label{xy} (E_{i,j}-E_{n+j,n+i})|_{\mathcal{B}}=E_{i,j}^x-E_{j,i}^y\qquad \text{for}~ i,j\in\overline{1,n}\en with
\eq E_{i,j}^x|_{\mathcal{B}}=\left\{\begin{array}{ll}-x_j\partial_{x_i}-\delta_{i,j}&\mbox{if}\;
i,j\in\overline{1,n_1},\\
\partial_{x_i}\partial_{x_j}&\mbox{if}\;i\in\overline{1,n_1},\;j\in\overline{n_1+1,n},\\
-x_ix_j &\mbox{if}\;i\in\overline{n_1+1,n},\;j\in\overline{1,n_1},\\
x_i\partial_{x_j}&\mbox{if}\;i,j\in\overline{n_1+1,n}
\end{array}\right.\en
and
\eq E_{i,j}^y|_{\mathcal{B}}=\left\{\begin{array}{ll}y_i\partial_{y_j}&\mbox{if}\;
i,j\in\overline{1,n_2},\\
-y_iy_j&\mbox{if}\;i\in\overline{1,n_2},\;j\in\overline{n_2+1,n},\\
\partial_{y_i}\partial_{y_j} &\mbox{if}\;i\in\overline{n_2+1,n},\;j\in\overline{1,n_2},\\
-y_j\partial_{y_i}-\delta_{i,j}&\mbox{if}\;i,j\in\overline{n_2+1,n}
\end{array}\right.\en
and
\eq E_{i,n+j}|_{\mathcal{B}}=\left\{\begin{array}{ll}
\ptl_{x_i}\partial_{y_j}&\mbox{if}\;
i\in\overline{1,n_1}, \,j\in\overline{1,n_2},\\
-y_j\ptl_{x_i}&\mbox{if}\;i\in\overline{1,n_1},\, j\in\overline{n_2+1,n},\\
{x_i}\partial_{y_j} &\mbox{if}\;i\in\overline{n_1+1,n}, \,j\in\overline{1,n_2},\\
-x_iy_j &\mbox{if}\; i\in \overline{n_1+1,n}, \,    j\in\overline{n_2+1,n}
\end{array}\right.\en

and
\eq\label{xy-} E_{n+i,j}|_{\mathcal{B}}=\left\{\begin{array}{ll}
-{x_j}{y_i}&\mbox{if}\;
j\in\overline{1,n_1}, \,i\in\overline{1,n_2},\\
-x_j\ptl_{y_i}&\mbox{if}\;j\in\overline{1,n_1}, \,i\in\overline{n_2+1,n},\\
{y_i}\partial_{x_j} &\mbox{if}\;j\in\overline{n_1+1,n}, \,i\in\overline{1,n_2},\\
\ptl_{x_j}\ptl_{y_i} &\mbox{if}\; j\in \overline{n_1+1,n}, \,     i\in\overline{n_2+1,n}.
\end{array}\right.\en

The related variated  Laplace operator becomes
\eq {\mathcal{D}}=\sum_{i=1}^{n_1}x_i\partial_{y_i}-\sum_{r=n_1+1}^{n_2}\partial_{x_r}\partial_{y_r}+\sum_{s=n_2+1}^n
y_s\partial_{x_s}.\en


 Set
\eq{\mathcal{B}}_{\langle
k'\rangle}=\mbox{Span}\{x^\alpha
y^\beta\mid\alpha,\beta\in\bb{N}\:^n,\sum_{r=n_1+1}^n\alpha_r-\sum_{i=1}^{n_1}\alpha_i+
\sum_{i=1}^{n_2}\beta_i-\sum_{r=n_2+1}^n\beta_r=k'\}\en for
$k'\in\bb{Z}$. Define
\eq{\mathcal{H}}_{\langle k'\rangle}=\{f\in
{\mathcal{B}}_{\langle k'\rangle}\mid
{\mathcal{D}}(f)=0\}.\en

The following is the first main theorem of this paper.

\begin{Thm}\label{main} For any $k'\in\bb{Z}$, if
the $\mathfrak{o}(2n,\bb{C})$-module  ${\mathcal{H}}_{\langle
k' \rangle}$ is irreducible, then it has the
Gelfand-Kirillov dimension
\eq \label{formula}
d=\left\{
    \begin{array}{ll}
      2n-1, &  \hbox{\text{\emph{if~}} $1=n_1< n_2<n-1$, \emph{or~} $3\leq n_1<n_2=n$,} \\
           & \hbox{\emph{or~}$1<n_1< n_2\leq n-1$,  \emph{or~}$n_1=n_2$ \emph{when~} $n\geq 5$;}\\
      2n-2, & \hbox{\emph{if~} $1=n_1< n_2=n-1,n$, \emph{or~} $2=n_1<n_2=n$} \\
            & \hbox{\emph{or~} $n_1=n_2$   \emph{when~} $n=4$;}\\
       2n-3, & \hbox{\emph{if~} $n_1=n_2$ \emph{when~} $n=2,3$.}
        \end{array}
  \right. \en
\end{Thm}

\begin{Rem}For this case, the minimal GK-dimension is $2n-3$. From our paper \cite{Bai-Hu}, the secondly minimal GK-dimension is max($4n-10, 2n-2$),
and the maximal GK-dimension is $\frac{n(n-1)}{2}$.  So $2n-1$ is larger than the secondly minimal GK-dimension of any unitary highest-weight modules when $n\leq 4$ and smaller then the secondly minimal GK-dimension of any unitary highest-weight modules when $n>4$.
\end{Rem}

We observe that the orthogonal Lie algebra
$$\mathfrak{o}(2n+1,\mbb{C})=\mathfrak{o}(2n,\mbb{C})\oplus\bigoplus_{i=1}^n
[\mbb{C}(E_{0,i}-E_{n+i,0})+\mbb{C}(E_{0,n+i}-E_{i,0})].$$
Let ${\mathcal B}'=\mbb{C}[x_0,x_1,...,x_n,y_1,...,y_n]$. We define a
non-canonical oscillator representation of $\mathfrak{o}(2n+1,\mbb{C})$ on
${\mathcal B}'$ by the differential operators in (\ref{xy})-(\ref{xy-}) and
$$ E_{0,i}|_{{\mathcal
B}'}=\left\{\begin{array}{ll}-x_0x_i&\mbox{if}\;i\in\ol{1,n_1},\\
x_0\ptl_{x_i}&\mbox{if}\;i\in\ol{n_1+1,n},\\
x_0\ptl_{y_{i-n}}&\mbox{if}\;i\in\ol{n+1,n+n_2},\\
-x_0y_{i-n}&\mbox{if}\;i\in\ol{n+n_2+1,2n}\end{array}\right.$$
and
$$ E_{i,0}|_{{\mathcal
B}'}=\left\{\begin{array}{ll}\ptl_{x_0}\ptl_{x_i}&\mbox{if}\;i\in\ol{1,n_1},\\
x_i\ptl_{x_0}&\mbox{if}\;i\in\ol{n_1+1,n},\\
y_{i-n}\ptl_{x_0}&\mbox{if}\;i\in\ol{n+1,n+n_2},\\
\ptl_{x_0}\ptl_{y_{i-n}}&\mbox{if}\;i\in\ol{n+n_2+1,2n}.\end{array}\right.$$
Now the variated Laplace operator becomes
$${\mathcal
D}'=\ptl_{x_0}^2-2\sum_{i=1}^{n_1}x_i\ptl_{y_i}+2\sum_{r=n_1+1}^{n_2}\ptl_{x_r}\ptl_{y_r}-2\sum_{s=n_2+1}^n
y_s\ptl_{x_s}.$$


Set
$${\mathcal B}'_{\langle k\rangle}=\sum_{i=0}^\infty {\mathcal B}_{\langle
k-i\rangle}x_0^i,\qquad {\mathcal H}'_{\langle k\rangle}=\{f\in {\mathcal B}'_{\langle
k\rangle}\mid {\mathcal D}'(f)=0\}.$$

The following is the second main theorem of this paper.

\begin{Thm}\label{main} For any $k'\in\bb{Z}$, the irreducible $\mathfrak{o}(2n+1,\bb{C})$-module  ${\mathcal{H}'}_{\langle
k' \rangle}$  has the
Gelfand-Kirillov dimension
\eq \label{formula}
d=\left\{
    \begin{array}{ll}
      2n, &  \hbox{\text{\emph{if~}} $1\leq n_1< n_2<n-1$, \emph{or~} $3\leq n_1<n_2=n$,} \\
           & \hbox{\emph{or~}$1<n_1< n_2= n-1$,  \emph{or~}$n_1=n_2$ \emph{when~} $n\geq 5$,}\\
          & \hbox{\emph{or~}$n_1=n_2=n=3$,  \emph{or~}$n_1=n_2>1$ \emph{when~} $n=4$;}\\
      2n-1, & \hbox{\emph{if~} $1=n_1< n_2=n-1,n$, \emph{or~} $2=n_1<n_2=n$}, \\
            & \hbox{\emph{or~} $n_1=n_2=2$   \emph{when~} $n=2,3$, \emph{or~} $n_1=n_2=1$   \emph{when~} $n=1,4$;}\\
       2n-2, & \hbox{\emph{if~} $1=n_1=n_2<n=2,3$.}
        \end{array}
  \right. \en
\end{Thm}

\begin{Rem}For this case, the minimal GK-dimension is $2n-2$. From our paper \cite{Bai-Hu}, the secondly minimal GK-dimension and the maximal GK-dimension are the same, i.e., $2n-1$.  So $2n$ is larger than the maximal GK-dimension of any unitary highest-weight modules.
\end{Rem}

The symplectic Lie
algebra\begin{eqnarray*}\hspace{1cm}\mathfrak{sp}(2n,\mbb{C})&=&
\sum_{i,j=1}^n\mbb{C}(E_{i,j}-E_{n+j,n+i})+\sum_{i=1}^n(\mbb{C}E_{i,n+i}+\mbb{C}E_{n+i,i})\\
& &+\sum_{1\leq i<j\leq n
}[\mbb{C}(E_{i,n+j}+E_{j,n+i})+\mbb{C}(E_{n+i,j}+E_{n+j,i})].\end{eqnarray*}
We define the two-parameter $\mathbb{Z}$-graded oscillator representation of $\mathfrak{sp}(2n,\mbb{C})$ on
${\mathcal B}$ via (\ref{xy})-(\ref{xy-}).

The related variated  Laplace operator becomes
\eq {{D}}=\sum_{r=n_1+1}^{n}x_r\partial_{x_r}-\sum_{i=1}^{n_1}x_i\partial_{x_i}+\sum_{i=1}^{n_2}y_i\partial_{y_i}
-\sum_{r=n_2+1}^{n}y_r\partial_{y_r}.\en
 Set
\eq{\mathcal{B}}_{\langle
k'\rangle}=\mbox{Span}\{x^\alpha
y^\beta\mid\alpha,\beta\in\bb{N}\:^n,\sum_{r=n_1+1}^n\alpha_r-\sum_{i=1}^{n_1}\alpha_i+
\sum_{i=1}^{n_2}\beta_i-\sum_{r=n_2+1}^n\beta_r=k'\}\en for
$k' \in\mathbb{Z}$. Then ${\mathcal{B}}_{\langle k'\rangle}=\{f\in{\mathcal{B}}\mid{D}(f)=k'f\}$.

 The following is the third main theorem of this paper.

\begin{Thm}\label{main} For any $k'\in\bb{Z}$, if
the $\mathfrak{sp}(2n,\bb{C})$-module  ${\mathcal{B}}_{\langle
k' \rangle}$ is irreducible, then it has the
Gelfand-Kirillov dimension
\eq
d=2n-1. \en
When $n_1=n_2$, the $\mathfrak{sp}(2n,\mbb{C})$-module ${\mathcal B}_{\langle 0\rangle}$ also has the
Gelfand-Kirillov dimension $d=2n-1.$ When $n_1=n_2=n$, the two irreducible components of ${\mathcal B}_{\langle 0\rangle}$ also have the
Gelfand-Kirillov dimension $d=2n-1.$
\end{Thm}

\begin{Rem}For this case, the minimal GK-dimension is $n$. From our paper \cite{Bai-Hu}, the secondly minimal GK-dimension and is $2n-1$.  So all the modules in the above theorem have the secondly minimal GK-dimension.
\end{Rem}

{\bf Acknowledgements.}The author is partially supported by 
This work is partially supported  by NSFC Grant No. 11601394 and the Fundamental Research Funds for the Central Universities Grant No. 2042016kf0041 from  Wuhan University. We would like to thank the referee for the comments on an earlier version
of this paper.


\section{Preliminaries on Gelfand-Kirillov Dimension}
We recall some  definitions and  properties of the Gelfand-Kirillov dimension.  Details may be found in Refs.\cite{Bo-Kr, Ja, Kr-Le, NOTYK, Vogan-78, Vogan-91}.

\begin{definition}
Let $A$ be an algebra (not necessarily associative) generated  by a
finite-dimensional subspace $V$. Let $V^n$ denote the linear span of
all products of length   at most $n$ in elements of $V$. The
\emph{Gelfand-Kirillov dimension} of $A$  is defined by:
$$GKdim(A) =  \limsup_{n\rightarrow \infty} \frac{\log\mathrm{dim}( V^{n} )}{\log n}.$$
\end{definition}

\begin{Rem}It is well-known that the above definition  is independent of  the choice of the
finite  dimensional  generating subspace $V$ (see Ref.\cite{Bo-Kr, Kr-Le}).
Clearly $GKdim(A)=0$ if and only if $\mathrm{dim}(A) <\infty$.
\end{Rem}
The notion of Gelfand-Kirillov dimension can be extended for left $A$-modules. In fact, we have the following definition.
\begin{definition} Let $A$ be  an algebra (not necessarily associative) generated  by a finite-dimensional subspace $V$. Let $M$ be a left $A$-module generated  by a finite-dimensional subspace $M_{0}$.  Let $V^n$ denote the linear span of all products of length  at most $n$  in elements of $V$.
The \emph{Gelfand-Kirillov dimension} $GKdim(M)$ of $M$  is defined by
$$GKdim(M) = \limsup_{n\rightarrow \infty}\frac{\log\mathrm{dim}( V^{n}M_{0} )}{\log n}.$$
\end{definition}

In particular,  let $\frk{g}$ be a complex Lie algebra. Let $A=\mathcal{U}(\frk g)$ be the enveloping algebra of $\frk g$, with the standard filtration given by $A_{n}=\mathcal{U}_{n}(\frk g)$, the subspace of $\mathcal{U}(\frk g)$ spanned by products of at most $n$-elements of $\frk g$. By the Poincar\'{e}-Birkhoff-Witt theorem (see Knapp \cite[Prop. 3.16]{Knapp}), the graded algebra $\text{gr} (\mathcal{U}(\frk g))$ is canonically isomorphic to the symmetric algebra $S(\frk g)$. Suppose $M$ is a $\mathcal{U}(\frk g)$-module generated  by a finite-dimensional subspace $M_{0}$. We set $M_{n}=\mathcal{U}_{n}(\frk g)M_{0}$.  Denote $\text{gr}M=\bigoplus\limits_{n=0}^{\infty}\text{gr}_{n}M$, where $\text{gr}_{n}M=M_{n}/M_{n-1}$. Then $\text{gr}M$ becomes a graded $S(\frk g)$-module. We denote $\dim(M_{n})$ by $\varphi_{M}(n)$. Then we have the following lemma.

\begin{Lem}\label{hi-se}
\emph{(Hilbert-Serre \cite[Chapter VII. Th.41]{Za-Sa} )}
 With the notations as above, there exists a unique polynomial $\tilde{\varphi}_{M}(n)$ such that $\varphi_{M}(n)=\tilde{\varphi}_{M}(n)$ for large $n$. The leading term of $\tilde{\varphi}_{M}(n)$ is $$\frac{c(M)}{(d_{M})!}n^{d_{M}},$$ where $c(M)$ is an integer.
\end{Lem}
\begin{Rem}

From the definition of Gelfand-Kirillov dimension, we know $$GKdim(M)=\limsup_{n\rightarrow \infty} \frac{\log\dim (U_{n}(\frk g)M_{0} )}{\log n}=\limsup_{n\rightarrow \infty} \frac{\log\tilde{\varphi}_{M}(n)}{\log n}=d_{M}=\dim \mathscr{V}(M).$$
\end{Rem}

\begin{ex}
Let $M=\mathbb{C}[x_{1},...,x_{k}]$. Then $M$ is an algebra generated by the  finite-dimensional subspace $V=Span_{\mathbb{C}}\{x_{1},...,x_{k}\}$.  So $M_{n}=V^{n}=\bigoplus\limits_{0\leq q \leq n}P_{q}[x_{1},...,x_{k}]$ is the subset of homogeneous polynomials of degree $\leq n$.
 Then \begin{align*}\varphi_{M}(n)=&\sum\limits_{0\leq q \leq n}\dim_{\mathbb{C}}(P_{q}[x_{1},...,x_{k}])\\
=&\sum\limits_{0\leq q\leq n}\binom{k+q-1}{q}\\
=&\binom{k+n}{n}\\
=&\frac{n^k}{k!}+O(n^{k-1}).\end{align*}
Then we have $GKdim(M)=k.$
\end{ex}

\section{Proof of the main theorem for $\mathfrak{o}(2n,\mbb{C})$}
We keep the same notations with the introduction. Through the paper we always take $\mathcal {K}=\sum\limits_{i,j=1}^{n}\mathbb{C}(E_{i,j}-E_{n+j,n+i})$, and $\mathcal{K}_{+}=\sum_{1\leq i<j\leq n}\mbb{C}(E_{i,j}-E_{n+j,n+i})$.
A weight vector $v$ in $\mathcal{B}$ is called a   \emph{$\mathcal{K}$-singular vector} if $\mathcal{K}_{+}(v)=0$.

We simply write  $E_{i,j}|_{\mathcal{B}}$ as $E_{i,j}$.
Take
\eq \mathfrak{h}=\sum_{i=1}^{n}\mbb{C}(E_{i,i}-E_{n+i,n+i})\en
as a Cartan subalgebra of $\mathfrak{o}(2n,\mbb{C})$ and the subspace spanned
by positive root vectors:
\eq \mathfrak{o}(2n,\mbb{C})_+=\sum_{1\leq i<j\leq n}\mbb{C}(E_{i,j}-E_{n+j,n+i})+\sum_{1\leq i<j\leq n}\mbb{C}(E_{i,n+j}-E_{j,n+i}).\en
Correspondingly, we have \eq \mathfrak{o}(2n,\mbb{C})_-=\sum_{1\leq i<j\leq n}\mbb{C}(E_{j,i}-E_{n+i,n+j})+\sum_{1\leq i<j\leq n}\mbb{C}(E_{n+j,i}-E_{n+i,j}).\en

If we take $\mathcal{P}_{+}=\sum_{1\leq i<j\leq n}\mbb{C}(E_{i,n+j}-E_{j,n+i})$, then $\mathfrak{o}(2n,\mbb{C})_+=\mathcal{K}_{+}+\mathcal{P}_{+}$. From the PBW theorem we know that the irreducible $\mathfrak{o}(2n,\bb{C})$-module ${\mathcal{H}}_{\langle k \rangle}=U(\mathfrak{g})v_{\mathcal{K}}=U(\mathfrak{g_{-}+\mathcal{P}_{+}})v_{\mathcal{K}}$ for any $\mathcal K$-singular vector $v_{\mathcal{K}}$. In the following we will compute the Gelfand-Kirillov dimension of $\mathcal{H}_{\langle k \rangle}$ in  a case-by-case way.

Firstly we need the following two well-known lemmas.
\begin{Lem}\emph{(Multinomial theorem)}\\
Let $n,m$ be two positive integers,then
\eq \left|\{(k_1,k_2,...,k_m)\in \bb{N}^{m}|\sum\limits_{i=1}^{m}k_{i}=n\}\right|={n+m-1 \choose m-1}.
\en
\end{Lem}

\begin{Lem}Let $p, n$ be two positive integers, then
$$\sum\limits_{i=0}^{n}i^p=\frac{(n+1)^{p+1}}{p+1}+\sum\limits_{k=1}^{p}\frac{B_k}{p-k+1}{p \choose k}(n+1)^{p-k+1},$$
where $B_k$ denotes a Bernoulli number.
\end{Lem}

From these two lemmas, we can get  several propositions.

\begin{Prop}\label{ak}Let $k\in \mathbb{N}$ and we denote $$M_{k}=\left\{\prod\limits_{\substack{1\leq i\leq n_1  \\ n_{1}+1\leq t\leq n}} ({x_i x_{t}})^{p_{it}}| \sum\limits_{\substack{1\leq i\leq n_1  \\ n_{1}+1\leq t\leq n}} p_{it}=k, p_{it}\in\mathbb{N} \right\}.$$ Then $$d_k=\dim Span_{\mathbb{R}}M_{k}={n_{1}+k-1 \choose k}{n-n_{1}+k-1 \choose k}\approx ak^{n-2},$$ for some constant a.
\end{Prop}
\begin{proof}From the definition of $M_k$, we know that all the elements in $M_k$ are monomials and they must form a basis for $Span_{\mathbb{R}}M_{k}$. Thus
\begin{align*}d_k=&\dim Span_{\mathbb{R}}M_{k}=\#\left\{\prod\limits_{\substack{1\leq i\leq n_1  \\ n_{1}+1\leq t\leq n}} ({x_i x_{t}})^{p_{it}}| \sum\limits_{\substack{1\leq i\leq n_1  \\ n_{1}+1\leq t\leq n}} p_{it}=k, p_{it}\in\mathbb{N} \right\}\\
=&\#\left\{\prod\limits_{\substack{1\leq i\leq n_1 }} (x_i)^{\sum_{n_{1}+1\leq t\leq n}p_{it}}\prod\limits_{\substack{n_{1}+1\leq t\leq n}} (x_t)^{\sum_{1\leq i\leq n_1 }p_{it}}| \sum\limits_{\substack{1\leq i\leq n_1  \\ n_{1}+1\leq t\leq n}} p_{it}=k, p_{it}\in\mathbb{N} \right\}\\
=&{n_{1}+k-1 \choose k}{n-n_{1}+k-1 \choose k}\approx ak^{n-2}, \emph{~~for some constant~} \emph{a}.
\end{align*}
\end{proof}

The idea of the proof for the following propositions are very simple: Denote ${\mathcal{B}}=\bb{C}[x_1,...,x_n,y_1,...,y_n]$. We define a partial order (i.e., dictionary order) on the monomials of ${\mathcal{B}}$:
$$x_{1}^{p_{1}}...x_n^{p_n}y_{1}^{p_{n+1}}...y_n^{p_{2n}}\preceq x_{1}^{p'_{1}}...x_n^{p'_n}y_1^{p'_{n+1}}...y_n^{p'_{2n}}$$
if there exists $1\leq m \leq 2n$, such that $p_i=p'_i$ for any $i<m$ and $p_m<p'_m$. We can also interchange the place of $x_1$ and some $x_j$ (or $y_j$), then define a similar partial order.

Suppose $I=\cup\{i\}$ is a given index set and  $P$ is a  set  of homogeneous polynomials which are products of some binomials $(f_i-g_i)^{a_i}$ ($f_i$ and $g_i$ are monomials of  degree $2$ in ${\mathcal{B}}$, and $\sum \limits_{i\in I}a_i=k$ is a constant), i.e., $P=\{\prod\limits_{i\in I}(f_i-g_i)^{a_i}| \sum a_i=k\}$. We fix $i\in \overline{1,n}.$ We choose two subsets $I_1$ and $I_2$ in $I$, such that $I_1\cap I_2=\emptyset$, and $I_1\cup I_2=\cup \{i\}=I$. Suppose $f_{i_p}$ is a multiple of $x_i$ and $g_{i_p}$ is not a multiple of $x_i$ when $i_p\in I_1$, and $g_{i_l}$ is a multiple of $x_i$ and $f_{i_l}$ is not a multiple of $x_i$ when $i_l\in I_2$.

Then we have $$\dim Span_{\mathbb{R}}P\geq \#\{\prod\limits_{i_p\in I_1}(f_{i_p})^{a_{i_p}}\cdot\prod\limits_{i_l\in I_2}(g_{i_l})^{a_{i_l}}|I=I_1\sqcup I_2, \sum a_{i_p}+\sum a_{i_l}=k\},$$
here we interchange the place of $x_1$ and $x_i$. So the monomial $\prod\limits_{i_p\in I_1}(f_{i_p})^{a_{i_p}}\cdot\prod\limits_{i_l\in I_2}(g_{i_l})^{a_{i_l}}$ which contain the largest power of $x_i$ is the leading term in the expression of $\prod\limits_{i\in I}(f_i-g_i)^{a_i}$.

\begin{Prop}\label{n=n}
\begin{enumerate}
              \item $(n_1=n_2=1)$ Let $k\in \mathbb{N}$ and we denote $$T_k=\left\{\prod\limits_{\substack{2\leq p<t\leq n}} ({x_p y_{t}}-x_{t}y_{p})^{g_{pt}}\cdot \prod\limits_{2\leq t \leq n}(x_1x_t-y_1y_t)^{g_{1t}}| \sum\limits_{1\leq p<t\leq n} g_{pt}=k, g_{pt}\in\mathbb{N} \right\}.$$
Then we have \begin{align*} d_k=\dim Span_{\mathbb{R}}T_{k}\approx\left\{
                                                   \begin{array}{ll}
                                                     c_0k^{2n-4}, & \text{\emph{if}~} {n=2 \mathrm{~or~} n=3;} \\
                                                     c_1k^{2n-3}, & \text{\emph{if}~} {n=4;} \\
                                                     c_2k^{2n-2}, & \text{\emph{if}~} {n\geq 5.}
                                                   \end{array}
                                           \right. \end{align*}
Here $c_0, c_1$ and $c_2$  are some positive constants which are independent of $k$.
              \item $(n_1=n_2=n-1)$ Let $k\in \mathbb{N}$ and we denote $$S_k=\left\{\prod\limits_{\substack{1\leq i<r\leq n-1}} ({x_i y_{r}}-x_{r}y_{i})^{f_{ir}}\cdot \prod\limits_{1\leq i \leq n-1}(x_ix_n-y_iy_n)^{f_{in}}| \sum\limits_{1\leq i<r\leq n} f_{ir}=k, f_{ir}\in\mathbb{N} \right\}.$$
Then we have \begin{align*} d_k=\dim Span_{\mathbb{R}}S_{k}\approx\left\{
                                                   \begin{array}{ll}
                                                     a_0k^{2n-4}, & \text{\emph{if}~} {n=2 \mathrm{~or~} n=3;} \\
                                                     a_1k^{2n-3}, & \text{\emph{if}~} {n=4;} \\
                                                     a_2k^{2n-2}, & \text{\emph{if}~} {n\geq 5.}
                                                   \end{array}
                                           \right. \end{align*}
Here $a_0, a_1$ and $a_2$  are some positive constants which are independent of $k$.
    \item $(n_1=n_2=n)$Let $k\in \mathbb{N}$ and we denote $$R_k=\left\{\prod\limits_{\substack{1\leq i<r\leq n}} ({x_i y_{r}}-x_{r}y_{i})^{f_{ir}}| \sum\limits_{1\leq i<r\leq n} f_{ir}=k, f_{ir}\in\mathbb{N} \right\}.$$
Then we have \begin{align*} d_k=\dim Span_{\mathbb{R}}R_{k}\approx\left\{
                                                   \begin{array}{ll}
                                                     b_0k^{2n-4}, & \text{\emph{if}~} {n=2 \mathrm{~or~} n=3;} \\
                                                     b_1k^{2n-3}, & \text{\emph{if}~} {n=4;} \\
                                                     b_2k^{2n-2}, & \text{\emph{if}~} {n\geq 5.}
                                                   \end{array}
                                           \right. \end{align*}
Here $b_0, b_1$ and $b_2$  are some positive constants which are independent of $k$.
\item  $(1<n_1=n_2<n-1)$Suppose $1<n_1<n-1$.
Let $k\in \mathbb{N}$ and we denote \begin{align*}U_k&=\left\{\prod\limits_{\substack{n_1+1\leq p<t\leq n}} ({x_p y_{t}}-x_{t}y_{p})^{g_{pt}}\cdot \prod\limits_{\substack{1\leq i<r\leq n_1}} ({x_i y_{r}}-x_{r}y_{i})^{g_{ir}}\right.\\
&\left.\cdot \prod\limits_{\substack{1\leq i\leq n_1\\ n_1+1\leq t \leq n}}(x_ix_t-y_iy_t)^{g_{1t}}| \sum\limits_{1\leq p<t\leq n} g_{pt}=k, g_{pt}\in\mathbb{N} \right\}.\end{align*}
Then we have \begin{align*} d_k=\dim Span_{\mathbb{R}}U_{k}\approx\left\{
                                                   \begin{array}{ll}
                                                   e_1k^{2n-3}, & \text{\emph{if}~} {n=4;} \\
                                                     e_2k^{2n-2}, & \text{\emph{if}~} {n\geq 5.}
                                                   \end{array}
                                           \right. \end{align*}
Here $e_0, e_1$ and $e_2$  are some positive constants which are independent of $k$.
\end{enumerate}

\end{Prop}

\begin{proof}The statements $(1)$ and $(2)$ are dual to each other. The proof of $(4)$ is similar to $(1)$. So we only need to give the proof for $(1)$ and $(3)$.

\text{Proof of $(1)$:}

When $n=3$, we have
\begin{align*}d_k&=\dim Span_{\mathbb{R}}T_{k}\\
&=\dim Span_{\mathbb{R}}\left\{ ({x_1 x_{2}}-y_{1}y_{2})^{g_{12}}({x_1 x_{3}}-y_{1}y_{3})^{g_{13}}(x_2y_3-x_3y_2)^{g_{23}}|\sum g_{pt}=k\right\}\\
&\geq \dim Span_{\mathbb{R}}\left\{({x_1 x_{2}})^{g_{12}}(x_1x_3)^{g_{13}}(x_2y_3)^{g_{23}}|\sum g_{pt}=k   \right\}\\
&=\dim Span_{\mathbb{R}}\left\{({ x_{1}})^{g_{12}+g_{13}}(y_3)^{g_{23}}\right.\\
&\quad\quad\quad\quad\quad\quad\cdot(x_2)^{g_{12}+g_{23}}(x_3)^{g_{13}}
\left.|\sum g_{pt}=k   \right\}\\
&\approx c_{00}k^2, \emph{~for some constant }c_{00}.
\end{align*}
On the other hand, we have $d_k=\dim Span_{\mathbb{R}}T_{k}\leq c_{01}k^{3-1}=c_{01}k^{2},$ for some positive constant $c_{01}$. So we must have $d_k=\dim Span_{\mathbb{R}}T_{k}\approx c_{0}k^{2}=c_{0}k^{2n-4},$ for some positive constant $c_{0}$.

 When $n=4$, we have
\begin{align*}d_k&=\dim Span_{\mathbb{R}}T_{k}\\
&=\dim Span_{\mathbb{R}}\left\{ ({x_1 x_{2}}-y_{1}y_{2})^{g_{12}}({x_1 x_{3}}-y_{1}y_{3})^{g_{13}}({x_1 x_{4}}-y_{1}y_{4})^{g_{14}}\right.\\
&\quad\quad\quad\quad\quad\quad \left.(x_2y_3-x_3y_2)^{g_{23}}(x_2y_4-x_4y_2)^{g_{24}}(x_3y_4-x_4y_3)^{g_{34}}|\sum g_{pt}=k\right\}\\
&\geq \dim Span_{\mathbb{R}}\left\{({y_1 y_{2}})^{g_{12}}(x_1x_3)^{g_{13}}(x_1x_4)^{g_{14}}(x_3y_2)^{g_{23}}(x_2y_4)^{g_{24}}(x_4y_3)^{g_{34}}|\sum g_{pt}=k   \right\}\\
&=\dim Span_{\mathbb{R}}\left\{({ y_{2}})^{g_{12}+g_{23}}(x_1)^{g_{13}+g_{14}}(y_3)^{g_{34}}(y_4)^{g_{24}}\right.\\
&\quad\quad\quad\quad\quad\quad\cdot({y_1 })^{g_{12}}(x_3)^{g_{13}+g_{23}}(x_4)^{g_{14}+g_{34}}(x_2)^{g_{24}}
\left.|\sum g_{pt}=k   \right\}\\
&\approx c_{10}k^5, \emph{~for some constant }c_{10}.
\end{align*}
On the other hand, we have $d_k=\dim Span_{\mathbb{R}}T_{k}\leq c_{11}k^{6-1}=c_{11}k^{5},$ for some positive constant $c_{11}$. So we must have $d_k=\dim Span_{\mathbb{R}}T_{k}\approx c_{1}k^{5}=c_{1}k^{2n-3},$ for some positive constant $c_{1}$.

When $n\geq 5$, we have\begin{align*}d_k&=\dim Span_{\mathbb{R}}T_{k}\\
&=\dim Span_{\mathbb{R}}
\left\{\prod\limits_{\substack{2\leq p<t\leq 4}} ({x_p y_{t}}-x_{t}y_{p})^{g_{pt}}\cdot \prod\limits_{2\leq t \leq 4}(x_1x_t-y_1y_t)^{g_{1t}}\right.\\
& \quad\quad\quad\quad\quad\quad\left.\cdot\prod\limits_{\substack{2\leq p<t\leq n \\ t\geq 5}} ({x_p y_{t}}-x_{t}y_{p})^{g_{pt}}\cdot \prod\limits_{ t \geq 5}(x_1x_t-y_1y_t)^{g_{1t}}   | \sum\limits_{1\leq p<t\leq n} g_{pt}=k, \right\}\\
&\geq \dim Span_{\mathbb{R}}\left\{({ y_{2}})^{g_{12}+g_{23}}(x_1)^{g_{13}+g_{14}}(y_3)^{g_{34}}(y_4)^{g_{24}}\prod\limits_{5\leq t\leq n}y_{t}^{g_{it}+g_{2t}}\prod\limits_{3\leq p<n}y_{p}^{\sum\limits_{p<t\leq n} g_{pt}}
\right.\\
&\quad\quad\quad\quad\quad\quad\cdot({y_1 })^{g_{12}+\sum\limits_{5\leq t\leq n} g_{1t}}(x_3)^{g_{13}+g_{23}}(x_4)^{g_{14}+g_{34}}(x_2)^{g_{24}+\sum\limits_{5\leq t\leq n} g_{2t}}\prod\limits_{5\leq t\leq n} x_{t}^{\sum\limits_{3\leq p<t}g_{pt}}\\
&\quad\quad\quad\quad\quad\quad\left.|\sum g_{pt}=k   \right\}\\
&\approx c_{20}k^{2n-2}, \emph{~for some constant }c_{20}.
\end{align*}

On the other hand, we have
\begin{align}\nonumber
d_k=&\dim Span_{\mathbb{R}}T_{k}\\\nonumber
\leq &\dim Span_{\mathbb{R}}\left\{\prod\limits ({x_1 })^{p_{1}}\prod({ y_{t}})^{q_{t}}\cdot\prod({ x_{t}})^{l_{t}}\prod(y_{1})^{f_{1}} |\right.\\\nonumber
&\quad\quad\quad\quad\quad\left. p_{1}+\sum\limits_{2\leq t\leq n} q_{t}=f_{1}+\sum\limits_{2\leq t\leq n} l_{t}=k \right\}\\
\approx & c_{21} k^{2n-2},  \emph{~for some constant }c_{21}.\nonumber
\end{align}\nonumber

So we must have $d_k=\dim Span_{\mathbb{R}}T_{k}\approx c_{2}k^{2n-2},$ for some positive constant $c_2$.

\hspace{1mm}

\text{Proof of $(3)$:}

When $n=2$, we have $d_k=\dim Span_{\mathbb{R}}R_{k}=\dim Span_{\mathbb{R}}\left\{ ({x_1 y_{2}}-x_{2}y_{1})^{k}\right\}=1$.

When $n=3$, we have
\begin{align*}d_k&=\dim Span_{\mathbb{R}}R_{k}\\
&=\dim Span_{\mathbb{R}}\left\{ ({x_1 y_{2}}-x_{2}y_{1})^{f_{12}}(x_1y_3-x_3y_1)^{f_{13}}(x_2y_3-x_3y_2)^{f_{23}}|f_{12}+f_{13}+f_{23}=k\right\}\\
&\geq \dim Span_{\mathbb{R}}\left\{({x_1 y_{2}})^{f_{12}}(x_3y_1)^{f_{13}}(x_2y_3)^{f_{23}}|f_{12}+f_{13}+f_{23}=k   \right\}\\
&={3+k-1 \choose k}\\
&\approx\frac{1}{2}k^2.
\end{align*}
On the other hand, we have $d_k=\dim Span_{\mathbb{R}}R_{k}\leq b_{00}k^{3-1}=b_{00}k^{2},$ for some positive constant $b_{00}$. So we must have $d_k=\dim Span_{\mathbb{R}}R_{k}\approx b_{0}k^{2}=b_{0}k^{2n-4},$ for some positive constant $b_{0}$.

When $n=4$, we have
\begin{align*}d_k&=\dim Span_{\mathbb{R}}R_{k}\\
&=\dim Span_{\mathbb{R}}\left\{\prod\limits_{\substack{1\leq i<r\leq 4}} ({x_i y_{r}}-x_{r}y_{i})^{f_{ir}}| \sum\limits_{1\leq i<r\leq 4} f_{ir}=k, f_{ir}\in\mathbb{N} \right\}\\
&\geq \dim Span_{\mathbb{R}}\left\{({x_1 y_{2}})^{f_{12}}  ({x_3 y_{1}})^{f_{13}} ({x_1 y_{4}})^{f_{14}} ({x_2 y_{3}})^{f_{23}} ({x_4 y_{2}})^{f_{24}} ({x_3 y_{4}})^{f_{34}}\right.\\
& \quad\quad\quad\quad\quad \quad \left.| \sum\limits_{1\leq i<r\leq 4} f_{ir}=k, f_{ir}\in\mathbb{N} \right\}\\
&=\dim Span_{\mathbb{R}}\left\{({x_1 })^{f_{12}+f_{14}}  ({x_3 })^{f_{13}+f_{34}}  ({x_4 })^{f_{24}} ({x_2})^{f_{23}}
 \cdot ({ y_{2}})^{f_{12}+f_{24}}  ({y_{1}})^{f_{13}} ({ y_{4}})^{f_{14}+f_{34}} ({y_{3}})^{f_{23}} \right.\\
& \quad\quad\quad\quad\quad \quad \left.| \sum\limits_{1\leq i<r\leq 4} f_{ir}=k, f_{ir}\in\mathbb{N} \right\}\\
&\approx b_{11}k^5, \emph{~for some constant }b_{11}.
\end{align*}
On the other hand, we have $d_k=\dim Span_{\mathbb{R}}R_{k}\leq b_{12}k^{6-1}=b_{12}k^{5},$ for some positive constant $b_{12}$. So we must have $d_k=\dim Span_{\mathbb{R}}R_{k}\approx b_{1}k^{5}=b_{1}k^{2n-3},$ for some positive constant $b_{1}$.

When $n\geq 5$, we have
\begin{align*}d_k&=\dim Span_{\mathbb{R}}R_{k}\\
&=\dim Span_{\mathbb{R}}\left\{\prod\limits_{\substack{1\leq i<r\leq 4}} ({x_i y_{r}}-x_{r}y_{i})^{f_{ir}}\cdot \prod\limits_{\substack{1\leq i<r\leq n \\ r\geq 5}}({x_i y_{r}}-x_{r}y_{i})^{f_{ir}} | \sum\limits_{1\leq i<r\leq n} f_{ir}=k, f_{ir}\in\mathbb{N} \right\}\\
&\geq \dim Span_{\mathbb{R}}\left\{\prod\limits_{\substack{1\leq i<r\leq 4}} ({x_i y_{r}}-x_{r}y_{i})^{f_{ir}}\cdot \prod\limits_{5\leq r\leq n}({x_1 y_{r}})^{f_{1r}}  ({x_2 y_{r}})^{f_{2r}} ({x_r y_{3}})^{f_{3r}} ({x_r y_{4}})^{f_{4r}} \right.\\
& \quad\quad\quad\quad\quad\quad \quad \left.| \sum\limits_{1\leq i<r\leq n} f_{ir}=k, f_{ir}\in\mathbb{N} \right\}\\
&=\dim Span_{\mathbb{R}}\left\{({x_1 })^{f_{12}+f_{14}+\sum f_{1r}}  ({x_3 })^{f_{13}+f_{34}}  ({x_4 })^{f_{24}} ({x_2})^{f_{23}+\sum f_{2r}}\prod\limits_{5\leq r\leq n}(x_r)^{f_{3r}+f_{4r}}\right.\\
&\quad\quad\quad\quad\quad \quad \quad\left. \cdot ({ y_{2}})^{f_{12}+f_{24}}  ({y_{1}})^{f_{13}} ({ y_{4}})^{f_{14}+f_{34}+\sum f_{4r}} ({y_{3}})^{f_{23}+\sum f_{3r}} \prod\limits_{5\leq r\leq n}(y_r)^{f_{1r}+f_{2r}}\right.\\
& \quad\quad\quad\quad\quad \quad \quad\left.| \sum\limits_{1\leq i<r\leq n} f_{ir}=k, f_{ir}\in\mathbb{N} \right\}\\
&\approx b_{21}k^{2n-2}, \emph{~for some constant }b_{21}.
\end{align*}

On the other hand, we have
\begin{align*}d_k&=\dim Span_{\mathbb{R}}R_{k}\\
&\leq \dim Span_{\mathbb{R}} \left\{\prod\limits_{\substack{1\leq i\leq n  \\1 \leq r\leq n }} {x_i}^{a_i} y_{r}^{b_{r}}
 | \sum a_{i}=\sum b_{r}=k \right\}\\
&\approx b_{22}k^{2n-2},\end{align*} for some positive constant $b_{22}$. So we must have $d_k=\dim Span_{\mathbb{R}}R_{k}\approx b_{2}k^{2n-2},$ for some positive constant $b_{2}$.
\end{proof}

\begin{Prop}\label{n<m}
\begin{enumerate}
  \item $(1<n_1<n_2=n-1)$Suppose $1<n_1< n-1$. Let $k\in \mathbb{N}$ and we denote
 \begin{align*} \nonumber  V_k&=\left\{\prod\limits(x_ix_s)^{p_{is}} \prod\limits(x_iy_s)^{l_{is}}\prod\limits(x_sy_n)^{u_{s}}\prod\limits(y_{s}y_{n})^{q_{s}}\prod\limits({x_{i} x_{n}}-y_{i}y_{n})^{h_{i}}\prod\limits({x_{i} y_{r}}-x_{r}y_{i})^{f_{ir}}|\right.\\
  &\quad \left.\sum\limits_{\substack{1\leq i\leq n_1  \\n_1+1\leq s\leq n-1}}p_{is}+\sum\limits_{\substack{1\leq i\leq n_1  \\n_1+1\leq s\leq n-1}}l_{is}+\sum\limits_{\substack{n_1+1\leq s\leq n-1 }}u_{s}+ \sum\limits_{\substack{n_1+1\leq s\leq n-1 }}q_{s}\right.\\
  &\quad\quad \left.+\sum\limits_{\substack{1\leq i\leq n_1 }}h_{i}+\sum\limits_{1\leq i<r\leq n_1}=k \right\},\\
  \end{align*}
 then
$$d_k=\dim Span_{\mathbb{R}}V_{k}\approx bk^{2n-2},$$ for some constant $b$.

\item $(1=n_1<n_2<n-1)$Suppose $1<n_2< n-1$. Let $k\in \mathbb{N}$ and we denote
\begin{align*}
 W_k&=\left\{\prod\limits(x_1x_s)^{p_{s}} \prod\limits(x_1y_s)^{l_{s}}\prod\limits(x_sy_t)^{u_{st}}\prod\limits(y_{s}y_{t})^{q_{st}}\prod\limits({x_{1} x_{t}}-y_{1}y_{t})^{h_{t}}\prod\limits({x_{p} y_{t}}-x_{t}y_{p})^{g_{pt}}|\right.\\
&\quad \left.\sum\limits_{2\leq s\leq n_2}p_{s}+\sum\limits_{2\leq s\leq n_2}l_{s}
+\sum\limits_{\substack{2\leq s\leq n_2 \\ n_2+1\leq t\leq n}}u_{st}+ \sum\limits_{\substack{2\leq s\leq n_2 \\ n_2+1\leq t\leq n }}q_{st}+\sum\limits_{\substack{n_2+1\leq t\leq n }}h_{t}+\sum\limits_{n_2+1\leq p<t\leq n}g_{pt}=k \right\},\\
\end{align*}
 then
$$d_k=\dim Span_{\mathbb{R}}W_{k}\approx \beta k^{2n-2},$$ for some constant $\beta$.
\item$(1<n_1<n_2=n)$

Suppose $1<n_1<n$.  Let $k\in \mathbb{N}$ and we denote
 \begin{align*} \nonumber  Z_k&=\left\{\prod\limits(x_ix_s)^{p_{is}}\prod\limits(x_{i}y_{s})^{l_{is}}\prod\limits({x_{i} y_{r}}-x_{r}y_{i})^{f_{ir}}|\right.\\
  &\quad\quad \left.\sum\limits_{\substack{1\leq i\leq n_1  \\n_1+1\leq s\leq n}}p_{is}+\sum\limits_{\substack{1\leq i\leq n_1  \\n_1+1\leq s\leq n}}l_{is}+ \sum\limits_{\substack{1\leq i<r\leq n_1  }}f_{ir}=k \right\},
  \end{align*}
then
\begin{align*} d_k=\dim Span_{\mathbb{R}}Z_{k}\approx\left\{
                                                   \begin{array}{ll}
                                                     \alpha_0k^{2n-3}, & \text{\emph{if}~} {n_1=2 <n;} \\
                                                     \alpha_1k^{2n-2}, & \text{\emph{if}~} {3\leq n_1<n.} \end{array}
\right. \end{align*}
Here $\alpha_0$ and $\alpha_1$  are some positive constants which are independent of $k$.
\item$(1<n_1<n_2<n-1)$

 Let $k\in \mathbb{N}$. Suppose $1<n_1< n_2<n-1$ and we denote
 \begin{align*} \nonumber  N^{\prime}_k&=\left\{\prod\limits(x_ix_s)^{p_{is}}\prod\limits(y_{s}y_{t})^{q_{st}}\prod\limits(x_iy_s)^{l_{is}}
\prod\limits(x_sy_t)^{u_{st}}\right.\\
& \quad  \left.\cdot \prod\limits({x_{i} x_{t}}-y_{i}y_{t})^{h_{it}}\prod\limits({x_{i} y_{r}}-x_{r}y_{i})^{f_{ir}}\prod\limits({x_{p} y_{t}}-x_{t}y_{p})^{g_{pt}}\right.\\
  &\quad \left. \mid \sum\limits_{\substack{1\leq i\leq n_1  \\n_1+1\leq s\leq n_2}}(p_{is}+l_{is})+\sum\limits_{\substack{n_1+1\leq s\leq n-1\\ n_2+1\leq t\leq n }}(u_{st}+q_{st})\right.\\
&\quad \left.+ \sum\limits_{\substack{1\leq i\leq n_1 \\ n_2+1\leq t\leq n}}h_{it}+\sum\limits_{1\leq i<r\leq n_1}f_{ir}+\sum\limits_{n_2+1\leq p<t\leq n_1}g_{pt}=k \right\},
\end{align*}
 then
$$d_k=\dim Span_{\mathbb{R}}N'_{k}\approx ck^{2n-2},$$ for some constant $c$.
\end{enumerate}

\end{Prop}

\begin{proof}The statements $(1)$ and  $(2)$ are dual to each other. Their proofs are similar to the proof of $(4)$. So we only need to prove $(3)$ and $(4)$.

\text{Proof of $(3)$:}
When $n_1=2$, we have

\begin{align*}
d_{k}&=\dim Span_{\mathbb{R}}\left\{\prod\limits(x_ix_s)^{p_{is}}\prod\limits(x_{i}y_{s})^{l_{is}}({x_{1} y_{2}}-x_{2}y_{1})^{f_{12}}|\right.\\
  &\quad\quad \quad \quad \quad \quad \left.\sum\limits_{\substack{1\leq i\leq 2  \\n_1+1\leq s\leq n}}p_{is}+\sum\limits_{\substack{1\leq i\leq 2  \\3\leq s\leq n}}l_{is}+ f_{12}=k \right\}\\
=&\dim Span_{\mathbb{R}}\left\{\prod\limits_{1\leq i\leq 2}(x_i)^{\sum\limits_{3\leq s\leq n}p_{is}+l_{is}}\cdot\prod\limits_{3\leq s\leq n}(x_s)^{\sum\limits_{1\leq i\leq 2}p_{is}}\prod\limits_{3\leq s\leq n}(y_{s})^{\sum\limits_{1\leq i\leq 2}l_{is}}\right.\\
&\quad\quad \quad \quad \quad \quad \left.\cdot
\prod\limits({x_{1} y_{2}}-x_{2}y_{1})^{f_{12}}
|\sum\limits_{\substack{1\leq i\leq 2  \\3\leq s\leq n}}p_{is}+\sum\limits_{\substack{1\leq i\leq 2  \\3\leq s\leq n}}l_{is}+ f_{12}=k \right\}\\
&\approx \alpha_{0}k^{2n-3}, \emph{~for some constant }\alpha_{0}.
\end{align*}

When $n_1>2$, we have

\begin{align*}
d_{k}&=\dim Span_{\mathbb{R}}\left\{\prod\limits(x_ix_s)^{p_{is}}\prod\limits(x_{i}y_{s})^{l_{is}}\prod\limits({x_{i} y_{r}}-x_{r}y_{i})^{f_{ir}}|\right.\\
&\quad\quad \left.\sum\limits_{\substack{1\leq i\leq n_1  \\n_1+1\leq s\leq n}}p_{is}+\sum\limits_{\substack{1\leq i\leq n_1  \\n_1+1\leq s\leq n}}l_{is}+ \sum\limits_{\substack{1\leq i<r\leq n_1  }}f_{ir}=k \right\}\\
\geq &\dim Span_{\mathbb{R}}\left\{\prod\limits_{1\leq i\leq n_1}(x_i)^{\sum\limits_{n_1+1\leq s\leq n}p_{is}+l_{is}}\cdot\prod\limits_{n_1+1\leq s\leq n}(x_s)^{\sum\limits_{1\leq i\leq n_1}p_{is}}\prod\limits_{n_1+1\leq s\leq n}(y_{s})^{\sum\limits_{1\leq i\leq n_1}l_{is}}\right.\\
& \quad \quad \quad \left.\cdot
\prod\limits_{2\leq i<r\leq n_1}(x_{r}y_{i})^{f_{ir}}\cdot (x_1y_{n_1})^{f_{1n_1}}\cdot\prod\limits_{2\leq r<n_1}(x_ry_1)^{f_{1r}}\right.\\
&\left.\quad\quad |\sum\limits_{\substack{1\leq i\leq n_1  \\n_1+1\leq s\leq n}}p_{is}+\sum\limits_{\substack{1\leq i\leq n_1  \\n_1+1\leq s\leq n}}l_{is}+ \sum\limits_{\substack{1\leq i<r\leq n_1  }}f_{ir}=k \right\}\\
&=\dim Span_{\mathbb{R}}\left\{\prod\limits_{1\leq i\leq n_1}(x_i)^{\sum\limits_{n_1+1\leq s\leq n}p_{is}+l_{is}}\cdot \prod\limits_{2<r\leq n_1}(x_{r})^{\sum\limits_{i<r}f_{ir}}\cdot (x_1)^{f_{1n_1}}\cdot\prod\limits_{2\leq r<n_1}(x_r)^{f_{1r}}
\right.\\
& \quad \quad \quad\quad\quad \left.\cdot\prod\limits_{n_1+1\leq s\leq n}(x_s)^{\sum\limits_{1\leq i\leq n_1}p_{is}}\prod\limits_{n_1+1\leq s\leq n}(y_{s})^{\sum\limits_{1\leq i\leq n_1}l_{is}}\cdot
\prod\limits_{2\leq i< n_1}(y_{i})^{\sum\limits_{i<r\leq n_1}f_{ir}}\right.\\
&\quad\quad\quad\quad\quad \left.\cdot\prod\limits_{2\leq r<n_1}(y_1)^{f_{1r}}\cdot (y_{n_1})^{f_{1n_1}}
|\sum\limits_{\substack{1\leq i\leq 2  \\3\leq s\leq n}}p_{is}+\sum\limits_{\substack{1\leq i\leq 2  \\3\leq s\leq n}}l_{is}+ f_{12}=k \right\}\\
&\approx \alpha_{10}k^{2n-2}, \emph{~for some constant }\alpha_{10}.
\end{align*}

On the other hand, we have
\begin{align}\nonumber
d_k=&\dim Span_{\mathbb{R}}Z_{k}\\\nonumber
\leq &\dim Span_{\mathbb{R}}\left\{\prod\limits ({x_i })^{p_{i}}\cdot\prod(x_s)^{a_s}\prod(y_s)^{b_s}\prod(y_{i})^{f_{i}} \right.\\\nonumber
&\quad\quad\quad\quad\quad\left. |\sum p_{i}=\sum f_{i}+\sum a_s+\sum b_s=k \right\}\\
\approx & \alpha_{11} k^{2n-2},  \emph{~for some constant }\alpha_{11}.\nonumber
\end{align}\nonumber

So we must have $d_k=\dim Span_{\mathbb{R}}Z_{k}\approx \alpha_{1}k^{2n-2},$ for some positive constant $\alpha_{1}$.

\text{Proof of $(4)$:}

When $n_1=2<n_2<n-1$,  we have
\begin{align*}&d_{k}=\dim Span_{\mathbb{R}}\left\{\prod\limits(x_ix_s)^{p_{is}}\prod\limits(y_{s}y_{t})^{q_{st}}\prod\limits(x_iy_s)^{l_{is}}
\prod\limits(x_sy_t)^{u_{st}}\right.\\
& \quad \quad\quad \quad \quad \quad\quad  \left.\prod\limits({x_{i} x_{t}}-y_{i}y_{t})^{h_{it}}\prod\limits({x_{i} y_{r}}-x_{r}y_{i})^{f_{ir}}\prod\limits({x_{p} y_{t}}-x_{t}y_{p})^{g_{pt}}| \right.\\
&\left.\quad\quad \quad \quad \quad \quad \quad\sum\limits p_{is}+\sum\limits q_{st}+ \sum\limits l_{is}+\sum\limits u_{st}+\sum\limits h_{it}+\sum\limits f_{ir}+\sum\limits g_{pt}=k \right\}\\
&\geq \dim Span_{\mathbb{R}}\left\{\prod\limits(x_ix_s)^{p_{is}}\prod\limits(y_{s}y_{t})^{q_{st}}\prod\limits(x_iy_s)^{l_{is}}
\prod\limits(x_sy_t)^{u_{st}}\right.\\
& \quad \quad\quad \quad \quad \quad\quad  \left.\prod\limits_{n_2+2\leq t\leq n}({x_{1} x_{t}})^{h_{1t}}\cdot(y_1y_{n_2+1})^{h_{1,n_2+1}}\cdot(x_2x_{n_2+1})^{h_{2,n_2+1}}\cdot\prod\limits_{n_2+2\leq t\leq n} (y_2y_t)^{h_{2t}}\right.\\
&\left.\quad\quad \quad \quad \quad \quad \quad
\prod\limits({x_{1} y_{2}}-x_{2}y_{1})^{f_{12}}\prod\limits(x_{t}y_{n_2+1})^{g_{n_2+1,t}}| \right.\\
&\left.\quad\quad \quad \quad \quad \quad \quad\sum\limits p_{is}+\sum\limits q_{st}+ \sum\limits l_{is}+\sum\limits u_{st}+\sum\limits h_{it}+\sum\limits f_{ir}+\sum\limits g_{pt}=k \right\}\\
&\geq \dim Span_{\mathbb{R}} \left\{\left(({x_2})^{h_{2,n_2+1}+\sum (p_{2s}+l_{2s})}(x_1)^{\sum (p_{1s}+l_{1s})+f_{12}}(\prod\limits_{n_2+2\leq t\leq n}(x_1)^{ h_{1t}})\right.\right.\\
&\quad \quad\quad\quad\quad\quad\left.\left.\prod\limits_{n_2+2\leq t\leq n} (y_t)^{h_{2t}} (y_{t})^{\sum (q_{st}+u_{st})}({y_{n_2+1}})^{h_{1,n_2+1}+\sum g_{n_2+1,t}}\right)\right.   \\
& \quad \quad\quad\quad\quad\quad \left. \cdot \left( \prod\limits(x_s)^{p_{1s}+p_{2s}+\sum u_{st}}(x_{n_2+1})^{h_{2,n_2+1}}(\prod\limits_{n_2+2\leq t\leq n}(x_t)^{h_{1t}+g_{n_2+1,t}})\right.\right.\\
&\quad \quad\quad\quad\quad\quad\left.\left.(\prod\limits(y_{s})^{\sum q_{st}+\sum l_{is}})({y_{1}})^{h_{1,n_2+1}}\prod\limits_{n_2+2\leq t\leq n}(y_2)^{ h_{2t}}y_{2}^{f_{12}}\right)\right.\\
&\quad \quad\quad\quad\quad\quad \left. | \sum\limits p_{is}+\sum\limits q_{st}+ \sum\limits l_{is}+\sum\limits u_{st}+\sum\limits h_{it}+\sum\limits f_{ir}+\sum\limits g_{pt}=k \right\}\\
&\approx c_0k^{2n-2}, \emph{~for some constant }c_0.
\end{align*}

On the other hand, we have
\begin{align}\label{d'k<}
d_k=&\dim Span_{\mathbb{R}}N'_{k}\\
\leq &\dim Span_{\mathbb{R}}\left\{\prod\limits ({x_i })^{p_{i}}\prod({ y_{t}})^{q_{t}}\cdot\prod(x_s)^{a_s}\prod(y_s)^{b_s}\prod({ x_{t}})^{l_{t}}\prod(y_{i})^{f_{i}} |\right.\\
&\quad\quad\quad\quad\quad\left. \sum p_{i}+\sum q_{t}=\sum l_{t}+\sum f_{i}+\sum a_s+\sum b_s=k \right\}\\
\approx & c_{00} k^{2n-2},  \emph{~for some constant }c_{00}.
\end{align}

So we must have $d_k=\dim Span_{\mathbb{R}}N'_{k}\approx ck^{2n-2},$ for some positive constant $c$.

When $n_1>2$, we have a similar argument.  And for these cases we still have $d_k=\dim Span_{\mathbb{R}}N'_{k}\approx ck^{2n-2},$ for some positive constant $c$.

\end{proof}

Next, we will compute the Gelfand-Kirillov dimensions of our modules in a case-by-case way.

Luo and Xu \cite{Luo-Xu-lie} proved that
 for any
$n_1-n_2+1-\delta_{n_1,n_2}\geq k'\in\mathbb{Z}$,
${\mathcal{H}}_{\langle k' \rangle}$ is an irreducible
$\mathfrak{o}(2n,\bb{C})$-module. Moreover, the
homogeneous subspace $\mathcal{
B}_{\langle k'\rangle}=\bigoplus_{i=0}^\infty\eta^i(\mathcal{
H}_{\langle k'-2i\rangle})$ is a direct sum of
irreducible submodules. The module ${\mathcal{H}}_{\langle k' \rangle}$ under the assumption is of highest-weight type only if $n_2=n$, in which case $x_{n_1}^{-k'}$ is a highest-weight
vector with weight
$-k'\lmd_{n_1-1}+(k'-1)\lmd_{n_1}+[(k'-1)\dlt_{n_1,n-1}-2k'\dlt_{n_1,n}]\lmd_n$.

\subsection{Case 1. $n_1+1\leq n_2$ \text{and} $n_{1}-n_{2}+1\geq k' \in \mathbb{Z}$.}

\hspace{1cm}

In this case we have:
\begin{align} \label{ri}(E_{r,i}-E_{n+i,n+r})|_{\mathcal{B}}&=-x_i\ptl_{x_r}-y_i\ptl_{y_r} &\text{for~}& 1\leq i<r\leq
n_1,\\
 \label{si}(E_{s,i}-E_{n+i,n+s})|_{\mathcal{B}}&=-{x_i x_{s}}-{y_i}\ptl_{y_{s}}  &\text{for} ~&i\in\overline{1,n_1},\;s\in\overline{n_1+1,n_2},
\\
\label{ti}(E_{t,i}-E_{n+i,n+t})|_{\mathcal{B}}&=-{x_i x_{t}}+y_{i}y_{t}   &\text{for} ~&i\in\overline{1,n_1},\;t\in\overline{n_2+1,n},\\
 \label{sj}(E_{s,j}-E_{n+j,n+s})|_{\mathcal{B}}&=x_s\ptl_{x_j}-y_j\ptl_{y_s}   &\text{for}~& n_1<
j<s\leq n_2,\\
 \label{ts}(E_{t,s}-E_{n+s,n+t})|_{\mathcal{B}}&=x_{t}\ptl_{x_{s}}+y_{s} y_{t}  &\text{for} ~&s\in\overline{n_1+1,n_2},\;t\in\overline{n_{2}+1,n},\\
 \label{tp}(E_{t,p}-E_{n+p,n+t})|_{\mathcal{B}}&=x_t\ptl_{x_p}+y_t\ptl_{y_p}&\text{for}~& n_2+1\leq p<t\leq
n,\\
\label{inr}(E_{i,n+r}-E_{r,n+i})|_{\mathcal{B}}&=\ptl_{x_i}\ptl_{y_r}-\ptl_{x_r}\ptl_{y_i}  &\text{for~}& 1\leq i<r\leq
n_1,\\
\label{nri}(E_{n+r,i}-E_{n+i,r})|_{\mathcal{B}}&=-{x_i}{y_r}+{x_r}{y_i}  &\text{for~}& 1\leq i<r\leq
n_1,\\
\label{ins}(E_{i,n+s}-E_{s,n+i})|_{\mathcal{B}}&=\ptl_{x_i}\ptl_{y_s}-{x_s}\ptl_{y_i}  &\text{for} ~&i\in\overline{1,n_1},\;s\in\overline{n_1+1,n_2},
\\
\label{nsi}(E_{n+s,i}-E_{n+i,s})|_{\mathcal{B}}&=-{x_i}{y_s}-{y_i}\ptl_{x_s}  &\text{for} ~&i\in\overline{1,n_1},\;s\in\overline{n_1+1,n_2},
\\
\label{int}(E_{i,n+t}-E_{t,n+i})|_{\mathcal{B}}&=-y_t\ptl_{x_i}-x_t\ptl_{y_i}   &\text{for} ~&i\in\overline{1,n_1},\;t\in\overline{n_2+1,n},\\
\label{nti}(E_{n+t,i}-E_{n+i,t})|_{\mathcal{B}}&=-x_i\ptl_{y_t}-y_i\ptl_{x_t}   &\text{for} ~&i\in\overline{1,n_1},\;t\in\overline{n_2+1,n},\\
\label{jns}(E_{j,n+s}-E_{s,n+j})|_{\mathcal{B}}&=x_j\ptl_{y_s}-x_s\ptl_{y_j}   &\text{for}~& n_1<j<s\leq n_2,\\
\label{njs}(E_{n+j,s}-E_{n+s,j})|_{\mathcal{B}}&=-y_s\ptl_{x_j}+y_j\ptl_{x_s}   &\text{for}~& n_1<j<s\leq n_2,\\
\label{snt}(E_{s,n+t}-E_{t,n+s})|_{\mathcal{B}}&=-x_s{y_t}-x_t\ptl_{y_s}        &\text{for} ~&s\in\overline{n_1+1,n_2},\;t\in\overline{n_{2}+1,n},\\
\label{nst}(E_{n+s,t}-E_{n+t,s})|_{\mathcal{B}}&=-\ptl_{x_s}\ptl_{y_t}-y_s\ptl_{x_t}    &\text{for} ~&s\in\overline{n_1+1,n_2},\;t\in\overline{n_{2}+1,n},\\
\label{pnt}(E_{p,n+t}-E_{t,n+p})|_{\mathcal{B}}&=-x_p{y_t}+x_t{y_p}       &\text{for}~& n_2+1\leq p<t\leq n,\\
\label{npt}(E_{n+p,t}-E_{n+t,p})|_{\mathcal{B}}&=-\ptl_{x_p}\ptl_{y_t}+\ptl_{x_t}\ptl_{y_p}    &\text{for}~& n_2+1\leq p<t\leq n.
\end{align}
Then the above root elements form a basis for the subalgebra $\mathfrak{g}(\mathcal{P}_{+})_{-}:=\mathfrak{o}(2n,\mbb{C})_-+\mathcal{P}_{+}$.

From Luo-Xu \cite{Luo-Xu-lie} we know that the $\mathcal{K}$-singular vectors in ${\mathcal H}_{\langle k'\rangle}$
are:
\eq x_{n_1}^{m_1}y_{n_2+1}^{m_2}\qquad\mbox{with}\;-(m_1+m_2)=k',\en
\eq x_{n_1+1}^{m_1}y_{n_2+1}^{m_2}\qquad\mbox{with}\;m_1-m_2=k',\en
\eq\label{n} x_{n_1}^{m_1}y_{n_2}^{m_2}\qquad\mbox{with}\;-m_1+m_2=k',\en
for all possible  $m_1,m_2\in\mathbb{N}$. When $n_1+1=n_2=n$, the $\mathcal{K}$-singular vectors in ${\mathcal H}_{\langle k'\rangle}$ are those in $(\ref{n})$.

Let $\mathfrak{g}_1$ be the subalgebra of $\mathfrak{o}(2n,\bb{C})$ spanned by the root vectors in the following set:$$I_1:=\{(\ref{ri}),(\ref{sj}),(\ref{tp}),(\ref{inr}),(\ref{ins}),(\ref{int}),(\ref{nti}),(\ref{jns}),(\ref{njs}),(\ref{nst}),(\ref{npt})\}.$$

Let $\mathfrak{g}_2$ be the subalgebra of $\mathfrak{o}(2n,\bb{C})$ spanned by the root vectors in the following set:$$I_2:=\{(\ref{si}),(\ref{ti}),(\ref{ts}),(\ref{nri}),(\ref{nsi}),(\ref{snt}),(\ref{pnt})\}.$$

So we get $U(\mathfrak{g}(\mathcal{P}_{+})_{-})=U(\mathfrak{g}_{2})U(\mathfrak{g}_{1}).$  From the construction of the root vectors we have the following lemma.

\begin{Lem}Every root vector in $\mathfrak{g}_1$ acts locally nilpotently on $\mathcal{H}_{\langle k' \rangle}$ and every root vector in $\mathfrak{g}_2$ acts torsion-freely (injectively) on $\mathcal{H}_{\langle k' \rangle}$.

\end{Lem}


We take a $\mathcal K$-singular vector $v_{\mathcal{K}}=x_{n_1}^{-k'}$, and set $M_{0}=U(\mathfrak{g}_1)x_{n_1}^{-k'}$.
  Then $M_{0}$  is finite-dimensional from the above lemma.

Thus
$${\mathcal{H}}_{\langle k' \rangle}=U(\mathfrak{g})v_{\mathcal{K}}=U(\mathfrak{g_{-}+\mathcal{P}_{+}})v_{\mathcal{K}}=U(\mathfrak{g}_{2})M_0.$$

Let $k$ be any positive integer. We want to compute $\mathrm{dim}(U_{k}(\mathfrak{g}_{2})M_{0})$, and then get the Gelfand-Kirillov dimension of $U(\mathfrak{g}_{2})M_{0}$.

 The cardinality of a set $A$ is usually denoted by $|A|$.

We use $E_{p}$ to stand for the root vector in the equation $(p)$. And denote
\begin{align*}N_{0}(k)=&\left\{\left(\prod E_{\ref{si}}^{p_{is}}\prod E_{\ref{ti}}^{h_{it}}\prod E_{\ref{ts}}^{q_{st}}\prod E_{\ref{nri}}^{f_{ir}}\prod E_{\ref{nsi}}^{l_{is}}\prod E_{\ref{snt}}^{u_{st}}\prod E_{\ref{pnt}}^{g_{pt}}\right)v_{\mathcal{K}}\right.\\
& \left.| \sum  p_{is}+\sum h_{it}+\sum q_{st}+\sum f_{ir}+\sum l_{is}+\sum u_{st}+\sum g_{pt}=k,\right.\\
 &\left.p_{is}, h_{it}, q_{st}, f_{ir},l_{is}, u_{st}, g_{pt}\in \mathbb{N} \right\}.\end{align*}

From the definition we know
\begin{align*}&\left(\prod E_{\ref{si}}^{p_{is}}\prod E_{\ref{ti}}^{h_{it}}\prod E_{\ref{ts}}^{q_{st}}\prod E_{\ref{nri}}^{f_{ir}}\prod E_{\ref{nsi}}^{l_{is}}\prod E_{\ref{snt}}^{u_{st}}\prod E_{\ref{pnt}}^{g_{pt}}\right)v_{\mathcal{K}}\\
=&\left(\prod (-{x_i x_{s}}-{y_i}\ptl_{y_{s}})^{p_{is}}\prod(-{x_{i} x_{t}}+y_{i}y_{t})^{h_{it}} \prod(y_{s}y_{t})^{q_{st}}\prod(-x_iy_r+x_ry_i)^{f_{ir}}\right.\\
&\left.\cdot\prod(-x_iy_s-y_i\partial_{x_s})^{l_{is}}\prod(-x_sy_t-x_t\partial_{y_s})^{u_{st}} \right)\cdot x_{n_1}^{-k'}\\
=&\left(\prod (-{x_i x_{s}})^{p_{is}}\prod(-{x_{i} x_{t}}+y_{i}y_{t})^{h_{it}} \prod(y_{s}y_{t})^{q_{st}}\prod(-x_iy_r+x_ry_i)^{f_{ir}}\right.\\
&\left.\cdot\prod(-x_iy_s)^{l_{is}}\prod(-x_sy_t)^{u_{st}} \right)\cdot x_{n_1}^{-k'}\\
&+\text{lower degree polynomials of }~y_{s} \text{~and~} x_s.
\end{align*}

Now we suppose $1<n_1<n_2<n-1$.  Then we must have $$\dim Span_{\mathbb{R}}N_{0}(m)\geq d_m=\dim Span_{\mathbb{R}}N'_{m}.$$

Using the same idea with the proof of Proposition \ref{n<m} (4), we can also get $$\dim Span_{\mathbb{R}}N_{0}(m)\leq d_{m}=\dim Span_{\mathbb{R}}N'_{m}.$$

Thus $\dim Span_{\mathbb{R}}N_{0}(m)= d_{m}=\dim Span_{\mathbb{R}}N'_{m}.$

Then using Proposition \ref{n<m} (4), we can get
\begin{align*}
&\dim Span_{\mathbb{R}}(\bigcup\limits_{0\leq m \leq k}N_{0}(m))\\
=& \sum\limits_{0\leq m \leq k} \dim Span_{\mathbb{R}}N'_{m}\\
=&c\sum\limits_{0\leq m\leq k}m^{2n-2}\\
=&c'k^{2n-1}
\end{align*}

Also we  have
$$
\dim Span_{\mathbb{R}}(\bigcup\limits_{0\leq m\leq k }N_{0}(m))\leq \mathrm{dim}(U_{k}(\mathfrak{g}_{2})M_{0})\leq \dim{M_0}\dim Span_{\mathbb{R}}(\bigcup\limits_{0\leq m\leq k }N_{0}(m)).
$$

Then from the definition, we know that the Gelfand-Kirillov dimension of $U(\mathfrak{g}_{2})M_{0}$ is
$$d=2n-1,  \text{if $1<n_1<n_2<n-1$.}$$

When $n_1=1<n_2<n-1$ or $1<n_1<n_2=n-1$, through a similar argument and using Proposition \ref{n<m} (1) and  \ref{n<m}(2),
we find that  the Gelfand-Kirillov dimension of $U(\mathfrak{g}_{2})M_{0}$ is $$d=2n-1.$$

When $1=n_1<n_2=n-1$ or $1=n_1<n_2=n$, through a similar argument and using Proposition \ref{ak},
we find that  the Gelfand-Kirillov dimension of $U(\mathfrak{g}_{2})M_{0}$ is $$d=2n-2.$$

When $1<n_1<n_2=n$, through a similar argument and using Proposition \ref{n<m} (3),
we find that  the Gelfand-Kirillov dimension of $U(\mathfrak{g}_{2})M_{0}$ is
$$d=\left\{
      \begin{array}{ll}
          2n-2, & \hbox{if $n_1=2<n_2=n$;}\\
          2n-1, & \hbox{if $3\leq n_1<n_2=n$.}
      \end{array}
    \right.
$$

Therefore, for this case $n_1+1\leq n_2$, the Gelfand-Kirillov dimension of $U(\mathfrak{g}_{2})M_{0}$ is
\begin{align*}d=\left\{
                   \begin{array}{ll}
                     2n-2, & \hbox{if $n_1=1, n_2=n-1$ ~or~ $n_1=1,2,n_2=n$;}\\
                     2n-1, & \hbox{if $n_1=1<n_2<n-1$ ~or~ $1<n_1<n_2\leq n-1$ ~or~ $3\leq n_1<n_2=n$.}
                   \end{array}
                 \right.
\end{align*}
\hspace{1cm}
\subsection{Case 2. $n_1=n_2$ \emph{and} $0\geq k'\in \mathbb{Z}$.}

\hspace{1cm}

In this case we have:
\begin{align} \label{ri-}(E_{r,i}-E_{n+i,n+r})|_{\mathcal{B}}&=-x_i\ptl_{x_r}-y_i\ptl_{y_r} &\text{for~}& 1\leq i<r\leq
n_1,\\
\label{ti-}(E_{t,i}-E_{n+i,n+t})|_{\mathcal{B}}&=-{x_i x_{t}}+y_{i}y_{t}   &\text{for} ~&i\in\overline{1,n_1},\;t\in\overline{n_1+1,n},\\
 \label{tp-}(E_{t,p}-E_{n+p,n+t})|_{\mathcal{B}}&=x_t\ptl_{x_p}+y_t\ptl_{y_p}&\text{for}~& n_1+1\leq p<t\leq
n,\\
\label{inr-}(E_{i,n+r}-E_{r,n+i})|_{\mathcal{B}}&=\ptl_{x_i}\ptl_{y_r}-\ptl_{x_r}\ptl_{y_i}  &\text{for~}& 1\leq i<r\leq
n_1,\\
\label{nri-}(E_{n+r,i}-E_{n+i,r})|_{\mathcal{B}}&=-{x_i}{y_r}+{x_r}{y_i}  &\text{for~}& 1\leq i<r\leq
n_1,\\
\label{int-}(E_{i,n+t}-E_{t,n+i})|_{\mathcal{B}}&=-y_t\ptl_{x_i}-x_t\ptl_{y_i}   &\text{for} ~&i\in\overline{1,n_1},\;t\in\overline{n_1+1,n},\\
\label{nti-}(E_{n+t,i}-E_{n+i,t})|_{\mathcal{B}}&=-x_i\ptl_{y_t}-y_i\ptl_{x_t}   &\text{for} ~&i\in\overline{1,n_1},\;t\in\overline{n_1+1,n},\\
\label{pnt-}(E_{p,n+t}-E_{t,n+p})|_{\mathcal{B}}&=-x_p{y_t}+x_t{y_p}       &\text{for}~& n_1+1\leq p<t\leq n,\\
\label{npt-}(E_{n+p,t}-E_{n+t,p})|_{\mathcal{B}}&=-\ptl_{x_p}\ptl_{y_t}+\ptl_{x_t}\ptl_{y_p}    &\text{for}~& n_1+1\leq p<t\leq n.
\end{align}
Then the above root elements form a basis for the subalgebra $\mathfrak{g}(\mathcal{P}_{+})_{-}:=\mathfrak{o}(2n,\mbb{C})_-+\mathcal{P}_{+}$.

Suppose $n_1=n_2<n-1$. From Luo-Xu \cite{Luo-Xu-lie} we know that the $\mathcal{K}$-singular vectors in ${\mathcal H}_{\langle k'\rangle}$
are:
\eq x_{n_1}^{m_1}y_{n_2+1}^{m_2}\qquad\mbox{with}\;-(m_1+m_2)=k',\en
\eq x_{n_1}^{-k'}\zeta_{1}^{m+1}\qquad\mbox{with}\;m\in \mathbb{N},\en
\eq\label{n-} y_{n_1+1}^{-k'}\zeta_{2}^{m+1}\qquad\mbox{with}\;m\in \mathbb{N},\en
where $\zeta_1=x_{n_1-1}y_{n_1}-x_{n_1}y_{n_1-1}$ and $\zeta_2=x_{n_1+1}y_{n_1+2}-x_{n_1+2}y_{n_1+1}$.


Let $\mathfrak{g}_1$ be the subalgebra of $\mathfrak{o}(2n,\bb{C})$ spanned by the root vectors in the following set:$$I_1:=\{(\ref{ri-}),(\ref{tp-}),(\ref{inr-}),(\ref{int-}),(\ref{nti-}),(\ref{npt-})\}.$$

Let $\mathfrak{g}_2$ be the subalgebra of $\mathfrak{o}(2n,\bb{C})$ spanned by the root vectors in the following set:$$I_2:=\{(\ref{ti-}),(\ref{nri-}),(\ref{pnt-})\}.$$

So we get $U(\mathfrak{g}(\mathcal{P}_{+})_{-})=U(\mathfrak{g}_{2})U(\mathfrak{g}_{1}).$

We take a $\mathcal K$-singular vector $v_{\mathcal{K}}=x_{n_1}^{-k'}$, and set $M_{0}=U(\mathfrak{g}_1)x_{n_1}^{-k'}$.
  Then $M_{0}$  is finite-dimensional.

Thus
$${\mathcal{H}}_{\langle k' \rangle}=U(\mathfrak{g})v_{\mathcal{K}}=U(\mathfrak{g_{-}+\mathcal{P}_{+}})v_{\mathcal{K}}=U(\mathfrak{g}_{2})M_0.$$

Let $k$ be any positive integer. We want to compute $\mathrm{dim}(U_{k}(\mathfrak{g}_{2})M_{0})$, and then get the Gelfand-Kirillov dimension of $U(\mathfrak{g}_{2})M_{0}$.

The argument for this  case is similar to  case 1, and using Proposition \ref{n=n},
we find that  the Gelfand-Kirillov dimension of $U(\mathfrak{g}_{2})M_{0}$ is
\begin{align}\label{n1=n2}d=\left\{
                   \begin{array}{ll}
                     2n-3, & \hbox{if $n=2, 3$;}\\
                     2n-2, & \hbox{if $n=4$;}\\
                     2n-1, & \hbox{if $n\geq 5$.}
                   \end{array}
                 \right.
\end{align}

When $n_1=n_2=n-1,n$, the arguments are similar with the above, and the conclusion is the same with \ref{n1=n2}.
\section{Proof of the main theorem for $\mathfrak{o}(2n+1,\mbb{C})$}

We keep the same notations with the introduction. We know
$$\mathfrak{o}(2n+1,\mbb{F})=\mathfrak{o}(2n,\mbb{C})\oplus\bigoplus_{i=1}^n
[\mbb{C}(E_{0,i}-E_{n+i,0})+\mbb{C}(E_{0,n+i}-E_{i,0})]$$
and $\mathcal{B}'=\mbb{C}[x_0,x_1,...,x_n,y_1,...,y_n]$.

Luo and Xu \cite{Luo-Xu-lie} proved that
 for any
$ k'\in\mathbb{Z}$,
${\mathcal{H}'}_{\langle k' \rangle}$ is an irreducible
$\mathfrak{o}(2n+1,\bb{C})$-module. Moreover, the
homogeneous subspace ${\mathcal B}'=\bigoplus_{k\in \mathbb{Z}}\bigoplus_{i=0}^\infty(\eta')^i({\mathcal H}'_{\langle k'\rangle})$ is a
decomposition of irreducible submodules.  The module ${\mathcal{H}'}_{\langle k' \rangle}$ under the assumption is of highest-weight type only if $n_2=n$, in which case $x_{n_1}^{-k'}$ is a highest-weight vector
with weight
$-k'\lmd_{n_1-1}+(k'-1)\lmd_{n_1}+[(k'-1)\dlt_{n_1,n-1}-2k'\dlt_{n_1,n}]\lmd_n.$

\subsection{Case 1. $n_1<n_2$ \emph{and} $k'\in \mathbb{N}$.}

\hspace{1mm}

The representation of $\mathfrak{o}(2n+1,\mbb{C})$ on $\mathcal{B}'$ by the differential operators in (\ref{ri})-(\ref{npt}) and $\mathcal{K}_{+}$ with $|_{\mathcal{B}}$ is replaced by $|_{\mathcal{B}'}$ and also contains the following:
\begin{align} \label{0i}(E_{0,i}-E_{n+i,0})|_{\mathcal{B}'}&=-x_0x_i-y_i\ptl_{x_0} &\text{for} ~&i\in\overline{1,n_1},\\
\label{0s}(E_{0,s}-E_{n+s,0})|_{\mathcal{B}'}&=x_0\ptl_{x_s}-y_s\ptl_{x_0}    &\text{for} ~&s\in\overline{n_1+1,n_2},\\
 \label{0t}(E_{0,t}-E_{n+t,0})|_{\mathcal{B}'}&=x_0\ptl_{x_t}-\ptl_{x_0}\ptl_{y_t}   &\text{for} ~&t\in\overline{n_2+1,n},\\
\label{0ni}(E_{0,n+i}-E_{i,0})|_{\mathcal{B}'}&=x_0\ptl_{y_i}-\ptl_{x_0}\ptl_{x_i}   &\text{for} ~&i\in\overline{1,n_1},\\
\label{0ns}(E_{0,n+s}-E_{s,0})|_{\mathcal{B}'}&=x_0\ptl_{y_s}-{x_s}\ptl_{x_0}   &\text{for} ~&s\in\overline{n_1+1,n_2},\\
\label{0nt}(E_{0,n+t}-E_{t,0})|_{\mathcal{B}'}&=-x_0{y_t}-{x_t}\ptl_{x_0}    &\text{for} ~&t\in\overline{n_2+1,n}.
\end{align}

Now we want to compute the Gelfand-Kirillov dimensions of  the $\mathfrak{o}(2n+1,\bb{C})$-module ${\mathcal{H}'}_{\langle k'\rangle}$ and  ${\mathcal{H}'}_{\langle -k'\rangle}$ for this case.
%
From Luo-Xu \cite{Luo-Xu-lie} we know that ${\mathcal{H}'}_{\langle k'\rangle}$ is an irreducible $\mathfrak{o}(2n+1,\mbb{C})$-submodule generated by $x_{n_1+1}^{k'}$,  and ${\mathcal{H}'}_{\langle -k'\rangle}$ is an irreducible $\mathfrak{o}(2n+1,\mbb{C})$-submodule generated by $x_{n_1}^{k'}$. Then similar to the computation of
$\mathfrak{o}(2n+1,\bb{C})$, the Gelfand-Kirillov dimension of ${\mathcal{H}'}_{\langle k'\rangle}$ is
\begin{align*}d=\left\{
                   \begin{array}{ll}
                     2n-1, & \hbox{if $2=n_1<n_2=n$ or $1=n_1<n_2=n-1,n$;}\\
                     2n, & \hbox{if $3\leq n_1<n_2=n$ or $1<n_1<n_2=n-1$ or $1\leq n_1<n_2<n-1$.}
                   \end{array}
                 \right.
\end{align*}
${\mathcal{H}'}_{\langle -k'\rangle}$   has the same Gelfand-Kirillov dimension with ${\mathcal{H}'}_{\langle k'\rangle}$.

\subsection{Case 2. $n_1=n_2$ \emph{and} $k'\in \mathbb{N}$.}

\hspace{1mm}

The representation of $\mathfrak{o}(2n+1,\mbb{C})$ on $\mathcal{B}'$ by the differential operators in (\ref{ri})-(\ref{npt}) and $\mathcal{K}_{+}$ with $|_{\mathcal{B}}$ is replaced by $|_{\mathcal{B}'}$ and also contains the following:
\begin{align} \label{0i-}(E_{0,i}-E_{n+i,0})|_{\mathcal{B}'}&=-x_0x_i-y_i\ptl_{x_0} &\text{for} ~&i\in\overline{1,n_1},\\
 \label{0t-}(E_{0,t}-E_{n+t,0})|_{\mathcal{B}'}&=x_0\ptl_{x_t}-\ptl_{x_0}\ptl_{y_t}   &\text{for} ~&t\in\overline{n_2+1,n},\\
\label{0ni-}(E_{0,n+i}-E_{i,0})|_{\mathcal{B}'}&=x_0\ptl_{y_i}-\ptl_{x_0}\ptl_{x_i}   &\text{for} ~&i\in\overline{1,n_1},\\
\label{0nt-}(E_{0,n+t}-E_{t,0})|_{\mathcal{B}'}&=-x_0{y_t}-{x_t}\ptl_{x_0}    &\text{for} ~&t\in\overline{n_2+1,n}.
\end{align}

Now we want to compute the Gelfand-Kirillov dimensions of  the $\mathfrak{o}(2n+1,\bb{C})$-module ${\mathcal{H}'}_{\langle k'\rangle}$ and ${\mathcal{H}'}_{\langle -k'\rangle}$ for this case.
%
From Luo-Xu \cite{Luo-Xu-lie} we know that ${\mathcal{H}'}_{\langle -k'\rangle}$ is an irreducible $\mathfrak{o}(2n+1,\mbb{C})$-submodule generated by $x_{n_1}^{k'}$, and ${\mathcal{H}'}_{\langle k'\rangle}$ is an irreducible $\mathfrak{o}(2n+1,\mbb{C})$-submodule generated by $T_{1}(y_{n_1}^{k'-1})$ (here $T_1=\sum\limits_{i=0}^{\infty}\frac{(-2)^ix_0^{2i+1}{
\mathcal{D}}^i}{(2i+1)!}$ and $\mathcal{D}=-\sum\limits_{i=1}^{n_1}x_i\ptl_{y_i}+\sum\limits_{s=n_1+1}^{n_2}\ptl_{x_s}\ptl_{y_s}-\sum\limits_{t=n_2+1}^n
y_t\ptl_{x_t}$). Then similar to the computation of
$\mathfrak{o}(2n+1,\bb{C})$, the Gelfand-Kirillov dimension of ${\mathcal{H}'}_{\langle -k'\rangle}$ is
\begin{align*}d=\left\{
                   \begin{array}{ll}
                     2n-2, & \hbox{if $1=n_1=n_2<n=2,3$;}\\
                     2n-1, & \hbox{if $n_1=n_2=2$ when $ n=2,3$ or $n_1=n_2=1$ when $n=1,4$;}\\
                     2n,   & \hbox{if $n_1=n_2=n=3$ or $2\leq n_1=n_2\leq n=4$ or $1\leq n_1=n_2\leq n$ when $n\geq 5$.}
                   \end{array}
                 \right.
\end{align*}
${\mathcal{H}'}_{\langle k'\rangle}$   has the same Gelfand-Kirillov dimension with ${\mathcal{H}'}_{\langle -k'\rangle}$.

\section{Proof of the main theorem for $\mathfrak{sp}(2n,\mbb{C})$}
We keep the same notations with the introduction.
Recall the symplectic Lie
algebra\begin{eqnarray*}\hspace{1cm}\mathfrak{sp}(2n,\mbb{C})&=&
\sum_{i,j=1}^n\mbb{C}(E_{i,j}-E_{n+j,n+i})+\sum_{i=1}^n(\mbb{C}E_{i,n+i}+\mbb{C}E_{n+i,i})\\
& &+\sum_{1\leq i<j\leq n
}[\mbb{C}(E_{i,n+j}+E_{j,n+i})+\mbb{C}(E_{n+i,j}+E_{n+j,i})].\end{eqnarray*}
Again we take the Cartan subalgebra
$\mathfrak{h}=\sum_{i=1}^n\mbb{C}(E_{i,i}-E_{n+i,n+i})$ and the subspace
spanned by positive root vectors
$$\mathfrak{sp}(2n,\mbb{C})_+=\sum_{1\leq i<j\leq n}[\mbb{C}(E_{i,j}-E_{n+j,n+i})
+\mbb{C}(E_{i,n+j}+E_{j,n+i})]+\sum_{i=1}^n\mbb{C}E_{i,n+i}.$$

Correspondingly, we have $$\mathfrak{sp}(2n,\mbb{C})_-=\sum_{1\leq i<j\leq n}[\mbb{C}(E_{j,i}-E_{n+i,n+j})
+\mbb{C}(E_{n+i,j}+E_{n+j,i})]+\sum_{i=1}^n\mbb{C}E_{n+i,i}.$$

Fix $1\leq n_1\leq n_2\leq n$. We have the following two-parameter $\mathbb{Z}$-graded oscillator representation of $\mathfrak{sp}(2n,\mbb{C})$ on $\mathcal{B}=\mathbb{C}[x_1,...,x_n,y_1,...,y_n]$ determined by $$(E_{i,j}-E_{n+j,n+i})|_{\mathcal{B}}=E_{i,j}^{x}-E_{j,i}^{y}.$$
In particular we have
\begin{align}
 \label{si3}(E_{s,i}-E_{n+i,n+s})|_{\mathcal{B}}&=-{x_i x_{s}}-{y_i}\ptl_{y_{s}}  &\text{for} ~&i\in\overline{1,n_1},\;s\in\overline{n_1+1,n_2},
\\
\label{ti3}(E_{t,i}-E_{n+i,n+t})|_{\mathcal{B}}&=-{x_i x_{t}}+y_{i}y_{t}   &\text{for} ~&i\in\overline{1,n_1},\;t\in\overline{n_2+1,n},\\
  \label{ts3}(E_{t,s}-E_{n+s,n+t})|_{\mathcal{B}}&=x_{t}\ptl_{x_{s}}+y_{s} y_{t}  &\text{for} ~&s\in\overline{n_1+1,n_2},\;t\in\overline{n_{2}+1,n},\\
\label{nri3}(E_{n+r,i}+E_{n+i,r})|_{\mathcal{B}}&=-{x_i}{y_r}-{x_r}{y_i}  &\text{for~}& 1\leq i<r\leq
n_1,\\
\label{nsi3}(E_{n+s,i}+E_{n+i,s})|_{\mathcal{B}}&=-{x_i}{y_s}+{y_i}\ptl_{x_s}  &\text{for} ~&i\in\overline{1,n_1},\;s\in\overline{n_1+1,n_2},\\
\label{snt3}(E_{s,n+t}+E_{t,n+s})|_{\mathcal{B}}&=-x_s{y_t}+x_t\ptl_{y_s}        &\text{for} ~&s\in\overline{n_1+1,n_2},\;t\in\overline{n_{2}+1,n},\\
\label{pnt3}(E_{p,n+t}+E_{t,n+p})|_{\mathcal{B}}&=-x_p{y_t}-x_t{y_p}       &\text{for}~& n_2+1\leq p<t\leq n,\\
\label{ini}(E_{n+i,i})|_{\mathcal{B}}&=-x_i{y_i} &\text{for}~& i\in\overline{1,n_1},\\
\label{tnt}(E_{t,n+t})|_{\mathcal{B}}&=-x_t{y_t} &\text{for}~& t\in\overline{n_2+1,n}.
\end{align}
Then the above root elements form a subalgebra for $\mathfrak{sp}(2n,\mbb{C})$, denoted by $\mathfrak{g}_2$. The remaining root elements form another subalgebra for $\mathfrak{sp}(2n,\mbb{C})$, denoted by $\mathfrak{g}_1$.

Luo and Xu \cite{Luo-Xu-lie} proved that
 for any
$ k'\in\mathbb{Z}$, when $n_1<n_2$ or $k'\neq 0$,
${\mathcal{B}}_{\langle k' \rangle}$ is an irreducible weight
$\mathfrak{sp}(2n,\bb{C})$-module. Moreover, the module ${\mathcal{B}}_{\langle k' \rangle}$ under the assumption is of highest-weight type only if $n_2=n$,  in which case for
$m\in\mbb{N}$, $x_{n_1}^{-m}$ is a highest-weight vector of ${\mathcal
B}_{\langle -m\rangle}$ with weight $-m\lmd_{n_1-1}+(m-1)\lmd_{n_1}$,
$x_{n_1+1}^{m+1}$ is a highest-weight vector of ${\mathcal B}_{\langle
m+1\rangle}$ with weight
$-(m+2)\lmd_{n_1}+(m+1)\lmd_{n_1+1}+(m+1)\dlt_{n_1,n-1}\lmd_n$ if
$n_1<n_2=n$,  and $y_n^{m+1}$ is a highest-weight vector of ${\mathcal B}_{\langle
m+1\rangle}$ with weight $(m+1)\lmd_{n-1}-2(m+1)\lmd_n$ if $n_1=n_2=n$. When $n_1=n_2$, the subspace ${\mathcal B}_{\langle 0\rangle}$ is a direct sum
of two irreducible weight $\mathfrak{sp}(2n,\mbb{C})$-submodules. If $n_1=n_2=n$, they
are highest-weight modules with a highest-weight vector $1$ of
weight $-2\lmd_n$ and with a highest-weight vector
$x_{n-1}y_n-x_ny_{n-1}$ of weight $(1-\dlt_{n,2})\lmd_{n-2}-4\lmd_n$,
respectively.

We take $\mathcal {K}=\sum\limits_{i,j=1}^{n}\mathbb{C}(E_{i,j}-E_{n+j,n+i})$, and $\mathcal{K}_{+}=\sum_{1\leq i<j\leq n}\mbb{C}(E_{i,j}-E_{n+j,n+i})$.
A weight vector $v$ in $\mathcal{B}$ is called a   \emph{$\mathcal{K}$-singular vector} if $\mathcal{K}_{+}(v)=0$.


From the PBW theorem we have
 $${\mathcal{B}}_{\langle k' \rangle}=U(\mathfrak{g})v_{\mathcal{K}}=U(\mathfrak{g}_{2})U(\mathfrak{g}_{1})v_{\mathcal{K}}$$ for any fixed $\mathcal K$-singular vector $v_{\mathcal{K}}$. If we denote $M_0:=U(\mathfrak{g}_{1})v_{\mathcal{K}}$, then $M_0$ is finite-dimensional and ${\mathcal{B}}_{\langle k' \rangle}=U(\mathfrak{g}_{2})M_0$.

Similar to the $\mathfrak{o}(2n,\mbb{C})$ case, we can compute the Gelfand-Kirillov dimension of $\mathcal{B}_{\langle k' \rangle}$ in  a case-by-case way. Actually the Gelfand-Kirillov dimension  is equal to $$2n-1$$ for any irreducible $\mathfrak{sp}(2n,\bb{C})$-module  $\mathcal{B}_{\langle k' \rangle}$.

\end{document}